%% file: MiSiVi_20.tex
\documentclass[a4paper,reqno]{amsart}

\input{commands.tex}

\usepackage{bbm}
\usepackage{mathtools}
\usepackage{url}

\newcommand{\w}{\eta}

\begin{document}

\numberwithin{equation}{section}

\title[Concentration of eigenfunctions]{Concentration of eigenfunctions of Schr\"odinger operators}

\author{Boris Mityagin}
\address[Boris Mityagin]{Department of Mathematics, The Ohio State University, 231 West 18th Ave,
Columbus, OH 43210, USA}
\email{mityagin.1@osu.edu, boris.mityagin@gmail.com}

\author{Petr Siegl}
\address[Petr Siegl]{School of Mathematics and Physics, Queen's University Belfast, University Road, Belfast, BT7 1NN, UK}
\email{p.siegl@qub.ac.uk}

\author{Joe Viola}
\address[Joe Viola]{Laboratoire de Math\'ematiques J.~Leray, UMR 6629 du CNRS, Universit\'e de Nantes, 2, rue de la Houssini\`ere, 44322 Nantes Cedex 03, France}
\email{Joseph.Viola@univ-nantes.fr}



\date{July 21, 2020}

\begin{abstract}
We consider the limit measures induced by the rescaled eigenfunctions of single-well Schr\"odinger operators. We show that the limit measure is supported on $[-1,1]$ and with the density proportional to $(1-|x|^\beta)^{-1/2}$ when the non-perturbed potential resembles $|x|^\beta$, $\beta >0$, for large $x$, and with the uniform density for super-polynomially growing potentials. We compare these results to analogous results in orthogonal polynomials and semiclassical defect measures.
\end{abstract}

\thanks{
P.S.~acknowledges the support of the \emph{Swiss National Science Foundation}, SNSF Ambizione grant No.~PZ00P2\_154786, till December 2017 and of the OSU for his stays there in October 2017 and May 2018. J.V.\ acknowledges the support of the R\'egion Pays de la Loire through the project EONE (\'Evolution des Op\'erateurs Non-Elliptiques). The authors acknowledge the Research in Paris stay at the Institute of Henri Poincar\'e, Paris, September 3-23, 2018, during which an essential part of the work has been done. 
We are grateful to our colleagues, Doron Lubinsky,  Georgia Tech., Atlanta, Georgia; Andrei Martinez-Finkelshtein, Baylor University, Waco, Texas; Paul Nevai, the Ohio State University, Columbus, Ohio; Gabriel Rivi\`ere, Universit\'e de Nantes, for helpful discussions on literature and related results. 
}

\maketitle

\section{Introduction}

Let $A$ be a Schr\"odinger operator acting in $L^2(\R)$ 
\begin{equation} \label{aho.def}
A = -\frac{\dd^2}{\dd x^2} + Q(x),
\end{equation}
where $Q$ is a real-valued, even, unbounded single-well potential. More precisely, we suppose that $Q=V+W$, where $V$ is a sufficiently regular single-well (see Assumptions~\ref{asm:V}) and $W$ is its possibly irregular perturbation (satisfying Assumption~\ref{asm:W} that guarantees that $W$ is small in a suitable sense). Our main condition on the potential is that $V$ satisfies  
\begin{equation}\label{asm:V*}
\exists \beta \in (0, \infty], \quad \forall x \in (-1,1), \quad \lim_{t \to +\infty} \frac{V(xt)}{V(t)} = \omega_\beta(x),
\end{equation}
where
\begin{equation}\label{omega.b.def}
\omega_\beta(x):= \begin{cases}
|x|^\beta, & \beta \in (0,\infty),
\\
0, & \beta = \infty.
\end{cases}
\end{equation}
As explained in \cite[Sec.~1.3]{Seneta-1976}, the existence of the limiting function in \eqref{asm:V*} already implies that $\omega_\beta$ is a power of $|x|$ or zero; functions $V$ satisfying \eqref{asm:V*} with $\beta<\infty$ are called \emph{regularly varying}.

It is well-known (also under much weaker assumptions on $Q$) that the operator $A$, defined via its quadratic form, is self-adjoint with compact resolvent, hence its spectrum is real and discrete. In fact, all eigenvalues $\{\lambda_k\}$ of $A$ are simple, thus they can be ordered increasingly and the corresponding eigenspaces are one-dimensional. Since the potential $Q$ is real, eigenfunctions $\{\psi_k\}$ related to $\{\lambda_k\}$ can be selected as real functions satisfying 
\begin{equation}
A \psi_k = \lambda_k \psi_k, \quad \|\psi_k\| =  \|\psi_k\|_{L^2(\R)}= 1, \quad k \in \N.
\end{equation}
These conditions do not determine $\psi_k$ uniquely, since $-\psi_k$ satisfies the same conditions; nonetheless, the squares $\{\psi_k^2\}$ are already uniquely determined. 

Let $x_{\la_k}$ be positive turning points of $V$ corresponding to eigenvalues $\{\la_k\}$, \ie~
\begin{equation}\label{xk.def}
V(x_{\la_k}) = \la_k,\quad  x_{\la_k} > 0, \quad k \in \N. 
\end{equation}
We define non-negative normalized measures on $\R$ induced by the eigenfunctions $\{\psi_k\}$ by 
\begin{equation}\label{nu.k.def}
\dd \mu_k := x_{\la_k} \, \psi_k(x_{\la_k} x)^2 \, \dd x, \quad x \in \R, \quad k \in \N.
\end{equation}
This rescaling transforms the classically forbidden region $\{x \, : \, V(x) > \lambda_k\}$ with (super)-exponential decay of $\psi_k$ to $\R \setminus [-1,1]$ while the rescaled functions $\psi_k(x_{\la_k}\cdot)$ oscillate in $[-1,1]$. 
Notice that we do not include $W$ in the definition of $x_{\la_k}$ and thus in the rescaling of eigenfunctions; the assumptions on the size of $W$ comparing to $V$, see Assumption~\ref{asm:W} and Proposition~\ref{prop:W}, allow for treating $W$ perturbatively.

In this paper, we prove that measures \eqref{nu.k.def} converges (as $k \to \infty$) to a limiting concentration measure supported on $[-1,1]$
\begin{equation}\label{lim.intro}
\dd \mu_*:=\frac{\Gamma(\frac 12 + \frac 1\beta)}{2 \pi^\frac 12 \Gamma(1+\frac 1\beta)} 
\frac{\mathbbm 1_{[-1,1]}(x)}{(1-\omega_\beta(x))^\frac 12} \, \dd x,
\end{equation}
see~Theorem~\ref{thm:lim.EF}. This generalizes the classical result for the harmonic oscillator, \ie~$Q(x)=x^2$, namely \emph{the arcsine law} for the concentration measure
\begin{equation}
\frac{1}{\pi} \frac{\mathbbm 1_{[-1,1]}(x)}{\sqrt{1-x^2}} \, \dd x 
\end{equation}
of the Hermite functions. Limiting measures of the type \eqref{lim.intro} were found for rescaled eigenfunctions with a different normalization for polynomial, possibly complex, potentials in \cite[Thm.~2]{Eremenko-2008-8}. The concentration of eigenfunctions is in particular used in estimates of norms of the spectral projections of non-self-adjoint Schr\"odinger operators obtained through conjugation, see~\cite{Mityagin-2017-272}, in particular, Section 3.

Notice that the condition \eqref{asm:V*} does not require $V$ to be a polynomial. For instance, the potentials below satisfy 
both technical Assumption~\ref{asm:V} and the condition \eqref{asm:V*}:
\begin{equation}\label{V.beta.gamma}
V(x)=|x|^{\alpha} \log (1+ x^2), \quad \alpha >0, 
\end{equation}
lead to the limit
\begin{equation}
\omega_\alpha(x)=|x|^\alpha, \quad x \in (-1,1),
\end{equation}
while for the fast-growing potentials
\begin{equation}\label{V.exp}
V(x)=\exp (|x|^\gamma), \quad \gamma >0,
\end{equation}
the limit reads
\begin{equation}
\omega_\infty(x)=0, \quad x \in (-1,1);
\end{equation}
the latter is not a special case, see Proposition~\ref{prop:nu2}.i). Moreover, one can include further, possibly irregular and unbounded perturbations $W$, see~Proposition~\ref{prop:W} for examples of admissible $W$. 

We emphasize that while the limiting function, if exists, is always homogeneous, this not required for $V$; see examples \eqref{V.beta.gamma} and \eqref{V.exp} above. Thus rescaling leads to a semi-classical operator only in very special cases; a relation of our result and so called semi-classical defect measures in these special cases can be found in Section~\ref{subsec:s-c} below.

This paper is organized as follows. Our results with precise assumptions are formulated in Section~\ref{sec:results} and they are proved in Section~\ref{sec:proofs} relying on asymptotic formulas for the eigenfunctions $\{\psi_k\}$ summarized in Section \ref{subsec:EF.sum}. In Section~\ref{sec:EF.proofs} we prove the asymptotic formulas following and slightly extending the ideas and results in the book \cite[\S22.27]{Titchmarsh-1958-book2} and in \cite{Giertz-1964-14}. Finally, in Section~\ref{sec:discussion} our results are compared to the existing literature in more detail.
\subsection{Notation}
Throughout the paper, we employ notations and results summarized in Section~\ref{subsec:EF.sum}. In particular, to avoid many appearing constants, for $a,b \geq0$, we write $a\ls b$ if there exists a constant $C>0$, independent of any relevant variable or parameter, such that $a\leq Cb$; the relation $a\gs b$ is introduced analogously. By $a\approx b$ it is meant that $a\ls b$ and $a\gs b$. The natural numbers are denoted by $\N = \{1,2,\dots\}$ and $\N_0=\N \cup \{0\}$.

\section{Assumptions and results}
\label{sec:results}

Our results are obtained under the following assumptions on the potential $Q = V + W$. The conditions on $V$, similar to those used in~\cite{Titchmarsh-1958-book2, Giertz-1964-14}, guarantee that $V$ is an even single-well potential with sufficient regularity to obtain convenient asymptotic formulas for eigenfunctions (associated with large eigenvalues) of the corresponding Schr\"odinger operator, see Section~\ref{subsec:EF.sum} and \ref{sec:EF.proofs} for details. The conditions on $W$ ensure that it is indeed a small perturbation which does not essentially affect the shape of the eigenfunctions.  

\begin{asm}\label{asm:V} 
Let $V: \R \to \R$ satisfy the following conditions.

\begin{enumerate}[i)]
\item $V \in C(\R) \cap C^2(\R \setminus \{0\})$ is even, 
\begin{equation}\label{V.lim}
\lim_{|x| \to +\infty} V(x) = +\infty,
\end{equation}
\item there exists $\xi_0>0$ such that $V \in C^3(\R\setminus [-\xi_0,\xi_0])$, 
\begin{equation}\label{V.pos}
V(x)>0, \ V'(x)>0, \qquad x \geq \xi_0, 
\end{equation}
and
\begin{equation}\label{V.int}
\frac{V'^2}{V^\frac52} \in L^1((\xi_0,\infty)), \qquad \frac{V''}{V^\frac32} \in L^1((\xi_0,\infty)),
\end{equation}
\item there exists $\nu \geq -1$ such that for all $x \geq \xi_0$
\begin{equation}\label{V.asm}
\begin{aligned}
V'(x) & \approx V(x) x^\nu,
\\
|V''(x)| & \ls V'(x) x^\nu, \quad |V'''(x)| \ls V'(x) x^{2\nu}.
\end{aligned}
\end{equation}
\hfill $\blacksquare$
\end{enumerate}
\end{asm}

Assumption~\ref{asm:V} is an extension of conditions in \cite[\S22.27]{Titchmarsh-1958-book2} where the case $\nu=-1$, \ie~polynomial-like potentials, is analyzed; conditions analogous to Assumption~\ref{asm:V} are used also in \cite{Krejcirik-2019-276,Arifoski-2020-52} where the resolvent estimates of non-self-adjoint Schr\"odinger operators are given. The assumptions of \cite{Giertz-1964-14} allow for fast growing potentials and are based on suitable restrictions of $V'''$, see \cite[Condition~2]{Giertz-1964-14}.

The first assumption \eqref{V.asm} implies there are two constants $0< c_1 \leq c_2 < \infty$ such that for all $x \geq \xi_0$
\begin{equation}\label{V.beh}
\begin{aligned}
x^{c_1} &\ls V(x) \ls x^{c_2},&&  \nu = -1,
\\
\exp(c_1 x^{\nu+1}) & \ls V(x) \ls \exp(c_2 x^{\nu+1}), && \nu > -1.
\end{aligned}
\end{equation}
This can be seen from (with $\xi_0\leq x_1 \leq  x_2$) 
\begin{equation}\label{nu.xt.ineq.new}
\log \frac{V(x_2)}{V(x_1)} = \int_{x_1}^{x_2} \frac{V'(s)}{V(s)} \, \dd s 
\approx \int_{x_1}^{x_2} s^\nu \, \dd s
= 
\begin{cases}
\displaystyle
\frac{x_2^{\nu+1}-x_1^{\nu+1}}{\nu+1}, & \nu>-1, 
\\[2mm]
\displaystyle \log \frac{x_2}{x_1}, & \nu=-1.
\end{cases}
\end{equation}
The crucial technical observation used frequently in the proofs is that \eqref{V.asm} imply that, for any $\eps \in (0,1)$ and all sufficiently large $x>0$, we have
\begin{equation}\label{V.Delta.new}
V^{(j)}(x + \Delta) \approx V^{(j)}(x), \qquad |\Delta| \leq \eps x^{-\nu}, \quad j=0,1,
\end{equation}
\ie~we have a control of how much $V$ and $V'$ varies over the intervals of size $x^{-\nu}$, see Lemma~\ref{lem:delta}. Assumptions \eqref{V.int} and \eqref{V.asm} also imply that 
\begin{equation}
\frac{V'(x)}{V(x)^\frac32} = o(1), \qquad x \to +\infty,
\end{equation}
see Lemma~\ref{lem:VV'}, which is almost optimal condition for the separation property of the domain of the self-adjoint Schr\"odinger operator $B=-\dd^2/\dd x^2 +V(x)$, namely, 
\begin{equation}
\Dom(B) = W^{2,2}(\R) \cap \{f \in L^2(\R) \, : \, Vf \in L^2(\R)\},
\end{equation}
see \cite{Everitt-1978-79,Evans-1978-80,Krejcirik-2017-221}; note that the separation property might be lost for $A$ due to the possibly irregular $W$.

The following proposition relates the parameter $\nu$ and the condition~\eqref{asm:V*}.

\begin{proposition}\label{prop:nu2}
	Let $V$ satisfy Assumption~\ref{asm:V}. 
	\begin{enumerate}[i)]
		\item If $\nu>-1$, then $V$ satisfies the condition \eqref{asm:V*} with $\beta = \infty$. 
		\item If $\nu=-1$ and $V$ satisfies the condition \eqref{asm:V*}, then $\beta \in (0,\infty)$.
	\end{enumerate}
\end{proposition}
\begin{proof}
	Let $x \in (0,1)$ be fixed. From \eqref{nu.xt.ineq.new}, we have that for all $t \geq \xi_0/x$ 
		\begin{equation}\label{nu.xt.ineq.old}
		\log \frac{V(t)}{V(xt)} \approx  
		\begin{cases}
		\displaystyle
		\frac{t^{\nu+1}}{\nu+1}(1-x^{\nu+1}), & \nu>-1, 
		\\[2mm]
		\displaystyle -\log x, & \nu=-1.
		\end{cases}
		\end{equation}
		Thus, if $\nu>-1$, we get that for every $x \in (0,1)$ 
		\begin{equation}
		\lim_{t \to +\infty} \frac{V(xt)}{V(t)} = 0.
		\end{equation}
		If $\nu =-1$ and the condition \eqref{asm:V*} holds, then for every $x \in (0,1)$
		\begin{equation}
		x^{\beta_1} \leq  \lim_{t \to +\infty} \frac{V(xt)}{V(t)} \leq x^{\beta_2} 
		\end{equation}
		where $\beta_1, \beta_2 \in (0,\infty)$ are independent of $x$.
\end{proof}
 
In the next step, we formulate a condition on the perturbation $W$ that guarantees that it is small in a suitable sense (arising in the proof of Theorem~\ref{thm:EF.asym}). The appearing weight $w_1^{-2}$ is naturally related with the main part of the potential $V$, although, the precise formula \eqref{w12.def} might seem more complicated to grasp. It includes the turning  point $x_\la$ of $V$, the quantity $a_\la$ (the value of $V'$ at the turning point) and a ``natural small region'' around the turning point (characterized by $\delta$ and $\delta_1$), see Section~\ref{subsec:EF.sum} for details. Examples of perturbations satisfying Assumption~\ref{asm:W} are given in Proposition~\ref{prop:W} below.
 
\begin{asm}\label{asm:W} 
Let $w_1$ be as in \eqref{w12.def} below. Let $W: \R \to \R$ be even, locally integrable and satisfy 
		\begin{equation}\label{JW.def}
		\cJ_W(\la):= \int_0^\infty \frac{W(s)}{w_1(s)^2} \, \dd s = o(1), \quad \la \to +\infty. 
		\end{equation}
\hfill $\blacksquare$
\end{asm}
\begin{proposition}\label{prop:W}
Let $V(x) = |x|^\beta$, $\beta>0$, and let $W = W_1+W_2$ where $\supp W_1$ is compact,  $W_1 \in L^1(\R)$, $W_2 \in L^\infty_{\rm \loc}(\R)$ and let $|W_2(x)| \ls |x|^\gamma$, $x \in \R$, for some $\gamma\in \R$. Then \eqref{JW.def} is satisfied if $\beta>2\gamma+2$. Moreover, if $\beta>1$, already $W_1 \in L^1(\R)$ suffices (one can omit the condition on the compactness of support of $W_1$). 
\end{proposition}
\begin{proof}
For all large $\la>0$, we get (let $\supp W_1 \subset [-x_1,x_1]$)
\begin{equation}
\int_0^\infty \frac{W_1(s)}{w_1(s)^2} \, \dd s  = \int_0^{x_1} \frac{W_1(s)}{(\la-s^\beta)^\frac12} \, \dd s
\ls \frac{\|W_1\|_{L^1}}{\la^\frac 12}.
\end{equation}
For $\beta>1$ and $W_1\in L^1(\R)$ without the condition on $\supp W_1$, one can use \eqref{V.xla} and \eqref{delta.est} to obtain
\begin{equation}
\int_0^\infty \frac{W_1(s)}{w_1(s)^2} \, \dd s  
\ls 
\frac1{a_\la^\frac13} \int_0^{\infty}  W_1(s) \, \dd s
\ls \la^\frac{1-\beta}{3\beta} \|W_1\|_{L^1}.
\end{equation}

Next, changing the integration variable $s = x_\la t$ and using \eqref{delta.est}, we get (with the assumption $\beta > 2\gamma+2$)
\begin{equation}
\begin{aligned}
\int_0^\infty \frac{W_2(s)}{w_1(s)^2} \, \dd s &\ls 
\int_0^1 \frac{W_2(s)}{w_1(s)^2} \, \dd s +
\frac{x_\la^{1+\gamma}}{\la^\frac 12}\int_0^{\infty} \frac{t^\gamma \, \dd t}{|1-t^\beta|^\frac12} + \frac{x_\la^\gamma(\delta+\delta_1)}{a_\la^\frac13} 
\\
& \ls \la^{-\frac12}+\la^\frac{2 \gamma+2-\beta}{2 \beta} + \la^\frac{3 \gamma-2-2\beta}{3 \beta} 
\ls \la^{-\frac12} + \la^\frac{2 \gamma+2-\beta}{2 \beta}.
\end{aligned}
\end{equation}
\end{proof}
Conditions on $W$ in Proposition~\ref{prop:W}, in particular $\beta>2\gamma+2$ or $W \in L^1(\R)$ when $\beta>1$, arise also in \cite{Mityagin-2019-139,Krejcirik-2019-276}, where the Riesz basis property of eigenfunctions, eigenvalue asymptotics and resolvent estimates are analyzed for \emph{complex} $W$.

Our main result reads as follows. 
\begin{theorem}\label{thm:lim.EF}
Let $Q =V + W$ where $V$ and $W$ satisfy Assumptions~\ref{asm:V} and \ref{asm:W}, respectively. Let $V$ satisfy in addition the condition \eqref{asm:V*} and let $\{\mu_k\}$, $\mu_*$ be as in \eqref{nu.k.def}, \eqref{lim.intro}, respectively. Let
\begin{equation}\label{f.growth}
\cF_V := \{ f \in L^\infty_{\rm loc}(\R) \,: \, \exists M \geq 0, \ f \exp (-M |V|^\frac12) \in L^\infty(\R)    \}. 
\end{equation}

Then, for every $f \in \cF_V$, we have

\begin{equation}\label{lim.nu.k}
\lim_{k \to \infty} \int_\R f(x) \, \dd \mu_k (x) =  
\frac{\Gamma(\frac 12 + \frac 1\beta)}{2 \pi^\frac 12 \Gamma(1+\frac 1\beta)} 
\int_{-1}^1  \frac{f(x)}{(1-\omega_\beta(x))^\frac 12}\, \dd x.
\end{equation}

Hence, in particular, the measures $\{\mu_k\}$ converge weakly to the limit measure $\mu_*$  as $k \to \infty$.
\end{theorem}

\subsection{Distribution of zeros}

We remark that the related result on the number of zeros of the eigenfunction $\psi_k$ in $[-\eps x_{\la_k}, \eps x_{\la_k}]$, $\eps \in (0,1]$, denoted by $N_k(\eps x_{\la_k})$, is
\begin{equation}\label{lim.N-intro}
\lim_{k \to \infty} 
\frac{N_k( \eps x_{\la_k} )}{k}
= \frac{\Gamma(\frac 32 + \frac 1\beta)}{\pi^\frac 12 \Gamma(1+\frac 1\beta)}  \int_{-\eps}^\eps (1-\omega_\beta(x))^\frac12 \, \dd x, \qquad \eps \in (0,1].
\end{equation}
This generalizes the classical results for the harmonic oscillator, \ie~$Q(x)=x^2$, namely the semi-circle law for the limiting distribution of the number of zeros of Hermite functions, 
\begin{equation}
\lim_{k \to \infty} 
\frac{N_k( \eps \sqrt{2k+1})}{k}
= \frac{2}{\pi}  \int_{-\eps}^\eps \sqrt{1-x^2} \, \dd x, \qquad \eps \in (0,1], 
\end{equation}
see \eg~\cite{Rakhmanov-1982-119,Gawronski-1987-50,Kuijlaars-1999-99}. A generalization of \eqref{lim.N-intro} for polynomial, possibly complex, potentials has been given in \cite{Eremenko-2008-8}. 

The distribution of zeros of eigenfunctions $\psi_k$, see \eqref{lim.N-intro}, is closely related to the distribution of eigenvalues of $A$ and it is essentially proved in \cite[Sec.~7]{Titchmarsh-1962-book1}. Indeed, without the perturbation $W$, \ie~$W=0$, the eigenvalues of $A$ satisfy
\begin{equation}\label{EV.asym}
\frac{\pi^\frac 12 \Gamma(1+\frac 1\beta)}{\Gamma(\frac 32 + \frac 1\beta)} x_{\la_k} \la_k^\frac 12  = \pi k(1+o(1)), \qquad k \to \infty,
\end{equation}
see \cite[Sec.~7]{Titchmarsh-1962-book1}, \cite[Thm.~2]{Giertz-1964-14}, so \eqref{lim.N-intro} follows from \cite[Lem.~7.3, Thm.~7.4]{Titchmarsh-1962-book1}. To include $W$, one could check that \eqref{EV.asym} remains valid for $V+W$, \eg~like in \cite[Thm.~6.6]{Mityagin-2019-139}, and adjust the arguments in \cite[Sec.~7]{Titchmarsh-1962-book1}. Alternatively, one can use the asymptotic formulas for $\{\psi_k\}$ and $\{\psi_k'\}$ in Section~\ref{subsec:EF.sum}; the latter can be 	derived by differentiating \eqref{y.int.eq}. The zeros of $\psi_k$ for $|x|<x_{\la_k}$ are in a neighborhood of the zeros of
\begin{equation}
J_\frac 13(\zeta(x)) + J_{-\frac 13}(\zeta(x)), \quad  \zeta(x)=\int_x^{x_{\la_k}} (\la - V(s))^\frac12 \, \dd s,
\end{equation}
and, for large $\zeta$, using asymptotic formulas for Bessel functions, see \cite[\S10.17]{DLMF}, these are in a neighborhood of zeros of
\begin{equation}
\sin \left(\zeta(x) + \frac \pi4 \right), \qquad |x|<x_{\la_k}.
\end{equation}

\section{The proofs}
\label{sec:proofs}

We start with an implication of the condition \eqref{asm:V*} for integrals frequently appearing in our analysis and proceed with the proof of Theorem~\ref{thm:lim.EF}.

\begin{lemma}\label{lem:Om.beta}
Let $V$ satisfy Assumption~\ref{asm:V} and the condition~\eqref{asm:V*}. Then, for every $g \in L^\infty((-1,1))$, 
\begin{equation}\label{Omega.lim}
\begin{aligned}
\lim_{t\to +\infty}  \int_{-1}^1 \left( 1 - \frac{V(x t)}{V(t)} \right)^{\frac12} g(x) \, \dd x
&= 
\int_{-1}^1 \left( 1 - \omega_\beta(x) \right)^{\frac12} g(x) \, \dd x,
\\
\lim_{t\to +\infty}  \int_{-1}^1 \left( 1 - \frac{V(x t)}{V(t)} \right)^{-\frac12} g(x) \, \dd x
& = 
\int_{-1}^1 \left( 1 - \omega_\beta(x) \right)^{-\frac12} g(x) \, \dd x.
\end{aligned}
\end{equation}
\end{lemma}
\begin{proof}
Both statements follow by \eqref{asm:V*} and the dominated convergence theorem. Since $V$ is even, it suffices to consider the integrals on $(0,1)$ only. 

First let $x \in [0,1/2]$ and let $\xi_0>0$ be as in Assumption~\ref{asm:V}. 
Since $V \in C(\R)$ and $V(y)$ is positive and increasing for $y \geq \xi_0$, see \eqref{V.pos}, we get that 
\begin{equation}\label{V.U.-1}
\begin{aligned}
\frac{V(x t)}{V(t)} &\leq \frac{\max_{0 \leq y \leq \xi_0}|V(y)| +  \max_{\xi_0 \leq y \leq \frac t2} V(y) } {V(t)}
\\
&
\leq 
\frac{V(\frac t2) } {V(t)} \left(1+ \frac{\max_{0 \leq y \leq \xi_0}|V(y)|}{V(\frac t2)} \right), \qquad t \geq 2\xi_0.
\end{aligned}
\end{equation}
Thus \eqref{V.lim} and \eqref{asm:V*} imply that there exists $\eps_0>0$ such that for all $x \in [0,1/2]$ and all $t>t_0$ with $t_0 \geq 2 \xi_0$ (independent of $x$) we have
\begin{equation}\label{V.U.-1.b}
\frac{V(x t)}{V(t)} \leq 1-\eps_0.
\end{equation}
Combining \eqref{V.U.-1.b} and the assumption that $V$ is eventually increasing on $\R_+$, see \eqref{V.pos}, we have that  $V(xt) \leq V(t)$ for all $x \in [0,1]$ and all $t>t_0$. Thus the existence of an integrable bound in the first limit follows.

For the second limit, we use inequalities \eqref{nu.xt.ineq.old}. These imply in particular that there is a constant $c>0$ (depending only on $\nu$) such that for all $x \in [1/2,1)$ and all $t \geq 2\xi_0$
\begin{equation}\label{V.U.ineq}
\frac{V(x t)}{V(t)} \leq \frac{U(xt)}{U(t)}, \quad \text{where} \quad U(x):=\begin{cases}
x^c, & \nu = -1,
\\[2mm]
\exp \left(c x^{\nu+1}\right), & \nu > -1.
\end{cases}
\end{equation}

For $\nu=-1$, combing \eqref{V.U.-1.b} and \eqref{V.U.ineq} for $x \in [\frac 12,1)$, we arrive at the integrable bound
\begin{equation}
\left(1 - \frac{V(x t)}{V(t)} \right)^{-\frac12} |g(x)| \leq
\begin{cases}
\eps_0^{-\frac12} |g(x)|, & x \in [0,\frac 12),
\\[2mm]
(1 - x^c)^{-\frac12} |g(x)|, & x \in [\frac 12,1). 
\end{cases}
\end{equation}
For $\nu>-1$, we show that for all $x \in [\frac 12,1]$ and all sufficiently large $t \geq 2\xi_0$ (independently of $x$) 
\begin{equation}\label{U.ineq}
1-\frac{U(xt)}{U(t)} \geq 1-x^{\nu+1}.
\end{equation}
To see this, we introduce $y=1-x^{\nu+1} \in [0,y_0]$ with $y_0=1-(1/2)^{\nu+1}<1$ and  $s=ct^{\nu+1}$. Then \eqref{U.ineq} holds if 
\begin{equation}\label{ey.ineq}
e^{sy}(1-y)-1 \geq 0
\end{equation}
for all $y \in [0,y_0]$ and all large $s>0$ (independently of $y$). Since $e^{sy} \geq 1+sy$, we get
\begin{equation}
e^{sy}(1-y)-1 \geq y(s(1-y)-1),
\end{equation}
thus \eqref{ey.ineq} holds if
\begin{equation}
s \geq \frac{1}{1-y_0}.
\end{equation}

Hence the sought integrable bound reads
\begin{equation}
\left(1 - \frac{V(x t)}{V(t)} \right)^{-\frac12} |g(x)| \leq
\begin{cases}
\eps_0^{-\frac12} |g(x)|, & x \in [0,\frac 12),
\\[2mm]
(1 - x^{\nu+1})^{-\frac12} |g(x)|, & x \in [\frac 12,1). 
\end{cases}
\end{equation}

\end{proof}

\subsection{Summary of properties of eigenfunctions of Schr\"odinger operators}
\label{subsec:EF.sum}

We summarize properties eigenfunctions of Schr\"odinger operators with even single-well potentials $Q=V+W$ satisfying Assumptions~\ref{asm:V} and \ref{asm:W}. The details and proofs are given in Section~\ref{sec:EF.proofs}; this slightly extends the reasoning in \cite[\S22.27]{Titchmarsh-1958-book2} and \cite{Giertz-1964-14}.


Since $Q$ is an even function by assumption, we can restrict ourselves to $(0,+\infty)$. Following the notations of \cite{Giertz-1964-14}, we introduce (for enough large $\la >0$)
\begin{equation}\label{u.def}
\begin{aligned}
V(x_\lambda) &= \lambda, \quad (x_\la>0) \\
a_\lambda & =V'(x_\lambda),
\\
\zeta &= \zeta(x,\la)= 
\begin{cases}
\displaystyle
\int_x^{x_\la} (\la - V(s))^\frac12 \, \dd s, & 0<x<x_\la, 
\\[4mm]
\displaystyle
\ii \int_{x_\la}^x (V(s)-\la)^\frac12 \, \dd s, & x > x_\la,
\end{cases}
\\
b&=b(x,\la)=\left(\frac{\zeta}{\zeta'} \right)^{\frac 12}, \quad \text{where} \quad 
\arg b = 
\begin{cases}
0, & x > x_\la,
\\[2mm]
\displaystyle
\frac \pi 2, & x < x_\la,
\end{cases}
\\
u & = u(x,\lambda) = b K_{\frac 13}(-\ii \zeta),
\\
v & = v(x,\lambda) = b I_{\frac 13}(-\ii \zeta); 
\end{aligned}
\end{equation}
here $K_{1/3}$, $I_{1/3}$ are modified Bessel functions of order $1/3$. 
Furthermore, we define 
\begin{equation}\label{kappa.def}
\kappa_\la := \int_{x_\la}^\infty \left(\frac{|V''(t)|}{V(t)^\frac 32} 
+ \frac{V'(t)}{V(t)^\frac 52} \right)\, \dd t.
\end{equation}

The functions $u$ and $v$ are known to be two linearly independent solutions of the differential equation
\begin{equation}
-f''+(V-\la) f = Kf,
\end{equation}
where 
\begin{equation}\label{K.def}
K=K(x,\la) = - \left(
\frac{b''}{b} + \frac{1}{9 b^4}
\right)
= \frac 14
\left(
\frac 59 \frac{\la-V}{\zeta^2} - \frac{V''}{\la-V} - \frac54 \frac{V'^2}{(\la-V)^2}
\right);
\end{equation}
moreover, the Wronskian of $u$ and $v$ satisfies
\begin{equation}\label{uv.Wron}
W[u,v](x)=u(x)v'(x)-v(x)u'(x)=1.
\end{equation}
The $L^2$-solution of Schr\"odinger equation $-y''+Qy=\la y$ is then found by solving the integral equation (obtained by variation of constants)

\begin{equation}\label{y.int.eq.new}
y(x) = u(x) + \int_x^\infty G(x,s) (K(s) + W(s)) y(s) \, \dd s,
\end{equation}
where $G(x,s) = u(x) v(s) - v(x) u(s)$, see Theorem~\ref{thm:EF.asym} and its proof in Section~\ref{sec:EF.proofs}.

Next, for $0\leq x < x_\la$, one gets 
\begin{equation}\label{u.Bessel}
u(x) = \frac{\pi}{\sqrt 3} \left| b \right|
\left(
J_\frac 13(\zeta) + J_{-\frac 13}(\zeta)
\right), \qquad v(x) = -|b| J_\frac 13(\zeta).
\end{equation}
The positive numbers $\delta$ and $\delta_1$ are defined by
\begin{equation}\label{delta.def}
\zeta(x_\lambda-\delta) = - \ii \zeta(x_\lambda + \delta_1) = 1
\end{equation}
and they satisfy
\begin{equation}\label{delta.est}
\delta + \delta_1 = o(x_\la^{-\nu}), \qquad 
\delta \approx \delta_1 \approx a_\la^{-\frac13}, 
\quad \la \to +\infty,
\end{equation}
see Lemma~\ref{lem:delta} and its proof for details.
As $\la \to +\infty$, we have
\begin{equation}\label{V.xla}
V(x_\la)-V(x_\la-\delta) \approx a_\la \delta \approx a_\la^{\frac 23}, 
\quad
V(x_\la+\delta_1)-V(x_\la) \approx a_\la \delta_1 \approx a_\la^{\frac 23},
\end{equation}
see~Lemma~\ref{lem:delta} below.

If $|x| < x_\lambda$ stays away from turning points, $\zeta$ is large and so it is useful to employ asymptotic formulas for Bessel functions with large argument, see \cite[\S10.17]{DLMF}.
%
In particular, one obtains
\begin{equation}\label{u^2.exp}
u^2(x) = \frac{\pi}{(\la-V(x))^\frac12} (1 + \sin 2 \zeta + R_1(\zeta)), \quad |x| < x_\lambda,
\end{equation}
where (see also \cite[Sec.~7]{Giertz-1964-14})
\begin{equation}\label{R1.est}
|R_1(\zeta)| = \BigO(\zeta^{-1}), \qquad \zeta \to + \infty.
\end{equation}

For the absolute values of $u$ and $v$, we have that, for all large enough $\la >0$,  
\begin{equation}\label{|u|.|v|.est}
|u(x)| \ls \big(w_1(x) w_2(x) \big)^{-1}, \qquad |v(x)| \ls  w_1(x)^{-1} w_2(x),  \qquad x >0,
\end{equation}
with the weights 
\begin{equation}\label{w12.def}
\begin{aligned}
w_1(x) &= 
\begin{cases}
|\la - V(x)|^{\frac14}, &  x \in (0,x_\la - \delta) \cup (x_\la + \delta_1, \infty),  \\[1mm]
a_\la^\frac16, &x \in [x_\la - \delta, x_\la + \delta_1], 
\end{cases}
\\
w_2(x) &=
\begin{cases}
1, &  x \in (0,x_\la + \delta_1],   \\[1mm]
e^{-\ii \zeta}, &x \in (x_\la + \delta_1,\infty),
\end{cases}
\end{aligned}
\end{equation}
see Lemma~\ref{lem:w12} below. Notice that $\arg \zeta(x) = \pi/2$ for $x > x_\lambda$ thus $|u(x)|$ is exponentially decreasing while $|v(x)|$ is allowed to be exponentially increasing as $x \to +\infty$. 

Next, from Assumption~\ref{asm:V} we obtain the following estimates, frequently occurring in our statements and proofs. 
\begin{lemma}	\label{lem:VV'}
	Let $V$ satisfy Assumption~\ref{asm:V} and let $x_\la$ and $a_\la$ be as in \eqref{u.def}. 
	Then, as $\la \to + \infty$,
\begin{equation}\label{kappa.lem}
\begin{aligned}
	\left(  
	\frac{x_\la^{2\nu}}{\la}
	\right)^\frac12 
	&\approx 
	\left(  
	\frac{x_\la^{3\nu}}{a_\la}
	\right)^\frac12 
	\approx
	\frac{V'(x_\la)}{V(x_\la)^\frac32} 
	\\
	&
	\ls \kappa_\la = \int_{x_\la}^\infty \left(\frac{|V''(t)|}{V(t)^\frac 32} + \frac{V'(t)^2}{V(t)^\frac 52} \right) \, \dd t = o(1), \qquad \la \to + \infty.
\end{aligned}
\end{equation}
\end{lemma}
\begin{proof}
	The claims follow from $V'(x) \approx V(x) x^\nu$ for $x$ sufficiently large, see \eqref{V.asm}, and 
	\begin{equation}
	\frac{V'(x_\la)}{V(x_\la)^\frac32} = - \int_{x_\la}^\infty \left( \frac{V'(t)}{V(t)^\frac32}\right)' \, \dd t 
	\end{equation}	
	together with \eqref{V.int}.
\end{proof}

Finally, we have that
\begin{equation}\label{u.norm}
\begin{aligned}
\int_0^{\infty} u(x)^2 \, \dd x 
&= 
\left(\int_0^{x_\la} \frac{\pi \, \dd x}{(\la - V(x))^\frac12} \right)
\left(
1 +  
\BigO
\left(\frac1{x_\la} + 
\left(\frac{x_\la^{3\nu}}{a_\la}\right)^\frac16   \frac{\log \frac{a_\la}{x_\la^{3\nu}}}{x_\la^{1+\nu}}  
\right)
\right)
\\
&= \left(\int_0^{x_\la} \frac{\pi \, \dd x}{(\la - V(x))^\frac12}\right) (1+o(1)), \quad \la \to +\infty,
\end{aligned}
\end{equation}	
see Lemma~\ref{lem:u^2} below.

The following theorem shows that the function $u$ is the main term in the asymptotic formula for eigenfunctions of the operator $A$ from \eqref{aho.def}. The proof is given at the end of Section~\ref{sec:EF.proofs}. One can check that the eigenvalues of $A$ are simple and eigenfunctions are even or odd functions (since $Q$ is assumed to be even). Thus the eigenvalues and eigenfunctions of $A$ can be found by determining $\la>0$ for which solutions $y$ in \eqref{yu.rel.ap} of the differential equation \eqref{y.ode} satisfy a Dirichlet ($y(0)=0$) or a Neumann ($y'(0)=0$) boundary condition at $0$.

\begin{theorem}\label{thm:EF.asym}
Let $Q=V+W$ where $V$ and $W$ satisfy Assumptions~\ref{asm:V} and \ref{asm:W}, respectively. Let $x_\lambda$ and $u$ be as in \eqref{u.def}, let $w_1$, $w_2$ be as in \eqref{w12.def}, let $\kappa_\la$ as in \eqref{kappa.def} and let $\cJ_W$ be as in \eqref{JW.def}. Then, for every sufficiently large $\la>0$, there is a solution of 
	\begin{equation}\label{y.ode}
	-y'' + (Q-\lambda) y =0
	\end{equation}
	on $(0,+\infty)$ such that
	\begin{equation}\label{yu.rel.ap}
	y = u + r,
	\end{equation}
	where
	\begin{equation}\label{r.est}
	|r(x)| \leq \frac{C(\la)}{w_1(x) w_2(x)} ,  \qquad x >0, 
	\end{equation}
	and 
	\begin{equation}
	C(\la) = \BigO(\la^{-\frac12} + \kappa_\la + \cJ_W(\la)) = o(1), \quad  \la \to + \infty.
	\end{equation}
	Moreover
	\begin{equation}\label{y.norm}
	\begin{aligned}
	&\int_0^{\infty} y^2(x) \, \dd x 
	\\
	& \qquad= 
	\left(\int_0^{x_\la} \frac{\pi \, \dd x}{(\la - V(x))^\frac12} \right)
	\left(
	1  
	+ C(\la)
	+
	\BigO
	\left(\frac1{x_\la}+
	\left(\frac{x_\la^{3\nu}}{a_\la}\right)^\frac16   \frac{\log \frac{a_\la}{x_\la^{3\nu}}}{x_\la^{1+\nu}}  
	\right)
	\right)
	\\
	& \qquad= 
	\left(\int_0^{x_\la} \frac{\pi \, \dd x}{(\la - V(x))^\frac12} \right)
	\left(
	1 + o(1)
	\right), \qquad \la \to +\infty. 
	\end{aligned}
	\end{equation}
\end{theorem}

\subsection{Proof of Theorem~\ref{thm:lim.EF}}

Since the eigenfunctions $\{\psi_k\}$ are even or odd, we consider only $x \in (0,\infty)$. We select the eigenfunctions $\{\psi_k\}$ such that
\begin{equation}
\psi_k(x) = \frac{y_k(x)}{\|y_k\|}  = \frac{u_k(x) + r_k(x)}{\|y_k\|} , \quad x>0,
\end{equation}
where $y_k=y(\cdot,\la_k)$, $u_k=u(\cdot,\la_k)$ and $r_k=y_k-u_k$, see Section~\ref{subsec:EF.sum} and in particular Theorem~\ref{thm:EF.asym}. Hence, the densities $\{\phi_k\}$ of the measures $\{\mu_k\}$, see \eqref{nu.k.def}, satisfy
\begin{equation}
\begin{aligned}
\phi_k(x) &= x_{\la_k} \, \psi_k(x_{\la_k} x)^2 
\\
&= x_{\la_k} \frac{u_k(x_{\la_k} x)^2 + 2 r_k(x_{\la_k} x) u_k(x_{\la_k} x ) + r_k(x_{\la_k} x)^2}{\|y_k\|^2}. 
\end{aligned}
\end{equation}

In the sequel, notations and results summarized in Section~\ref{subsec:EF.sum} are used, moreover, we introduce the constant (for $\beta \in (0,\infty]$)
\begin{equation}\label{Omega.def}
\begin{aligned}
\Omega'_\beta&:=\int_{-1}^1 (1-|t|^\beta)^{-\frac12} \, \dd t = \frac{2 \pi^\frac 12 \Gamma(1+\frac 1\beta)}{\Gamma(\frac 12 + \frac 1\beta)}.
\end{aligned}
\end{equation} 
%
We also drop the subscript $k$ and work with quantities like $y=y(\cdot,\la)$ as $\la \to +\infty$.

First, Lemma~\ref{lem:Om.beta}, \eqref{y.norm} and the change of integration variables $x=x_{\la} t$ imply
\begin{equation}
\|y\|^2 = 2 \left(\int_0^{x_{\la}} \frac{\pi \, \dd x}{(\la - V(x))^\frac12} \right)
( 1 + o(1))
= \frac{\pi \Omega_\beta'  x_{\la}}{\la^{\frac12}}  (1+o(1))
, \quad \la \to +\infty.
\end{equation}

Thus with $f \in \cF_V$, see \eqref{f.growth}, and the change of integration variables, we get 
\begin{equation}\label{phik.f.int}
\int_0^\infty  \phi(x) f(x) \, \dd x = \frac{1}{\pi \Omega_\beta'} \frac{\la^\frac12}{x_{\la}} \left(\int_0^\infty y(x)^2 f \left(\frac x{x_{\la}} \right) \; \dd x \right)
(1+o(1)), \ \ \la \to +\infty;
\end{equation}
the integral indeed converges for $f \in \cF_V$ as can be seen from \eqref{f.exp.1}, \eqref{V.zeta} below and the behavior of $y$ at infinity, see \eqref{yu.rel.ap}, \eqref{r.est}, \eqref{|u|.|v|.est} and \eqref{w12.def}.

First we show that the contribution from the region around the turning point is negligible. It follows from \eqref{delta.est} and \eqref{kappa.lem} that
\begin{equation}
\frac{\delta_1}{x_\la} \approx 
\left(
\frac{x_\la^{3\nu}}{a_\la} 
\right)^\frac13
\frac{1}{x_\la^{\nu+1}} = o(1), \quad \la \to +\infty, 
\end{equation}
hence, since $f \in L^\infty_{\rm loc}(\R)$,
\begin{equation}\label{f.sup}
\esssup_{0 \leq x \leq x_\la+\delta_1} \left|f \left(\frac x{x_{\la}} \right) \right| = \BigO(1), \quad \la \to +\infty.
\end{equation}
Employing estimates \eqref{|u|.|v|.est}, \eqref{r.est}, \eqref{f.sup} and \eqref{delta.est} in the last step, we obtain
\begin{equation}
\begin{aligned}
\cI_1:=\frac{\la^\frac12}{x_{\la}} \int_{x_{\la}-\delta}^{x_{\la}+\delta_1} y(x)^2 \left|f \left(\frac x{x_{\la}} \right)\right| \; \dd x 
\ls
\frac{\la^\frac12}{x_{\la}} \frac{(1 + C(\la)^2)(\delta + \delta_1)}{a_{\la}^\frac13} 
\ls \frac{\la^\frac12}{x_{\la} a_{\la}^{\frac 23}}. 
\end{aligned}
\end{equation}

Similarly, since $x_\la^{-\nu} \leq x_\la$ and $\delta_1 = o(x_\la^{-\nu})$ as $\la \to +\infty$, see \eqref{delta.est}, we get (using \eqref{|u|.|v|.est}, \eqref{V.xla} and changing the integration variables $-\ii \zeta(x) = |\zeta(x)|=t$)
\begin{equation}
\begin{aligned}
\cI_2&:=\frac{\la^\frac12}{x_{\la}} \int_{x_{\la}+\delta_1}^{x_{\la}+\frac{x_{\la}^{-\nu}}2} y(x)^2 \left|f \left(\frac x{x_{\la}} \right)\right| \; \dd x 
\\
&
\ls
\frac{\la^\frac12}{x_{\la}}
 \int_{x_{\la}+\delta_1}^{x_{\la}+\frac{x_{\la}^{-\nu}}2} 
\frac{(1+C(\la)^2)e^{-2 |\zeta(x)|}}{(V(x) - \la)^\frac 12} \; \dd x
\ls
\frac{\la^\frac12}{x_{\la} a_\la^\frac23}
\int_{1}^{\infty} 
e^{-2t} \; \dd t
\ls \frac{\la^\frac12}{x_{\la} a_{\la}^{\frac 23}}. 
\end{aligned}
\end{equation}

We investigate the region $(x_{\la}+x_{\la}^{-\nu}/2,\infty)$ and also explain the convergence of the integral in \eqref{phik.f.int}. To this end, we recall that by assumption $f \in \cF_V$, see \eqref{f.growth}, thus
with some $M>0$
\begin{equation}\label{f.exp.1}
\begin{aligned}
&\left|f \left(\frac x{x_{\la}} \right)\right| \exp(-|\zeta(x)|)
\\
&\quad \leq \|f \exp(-M|V|^\frac12)\|_{L^\infty} \exp \left(-|\zeta(x)|\Big(1- M \frac{
	\big|V(\frac{x}{x_\la})\big|^\frac12}{|\zeta(x)|} \Big) \right)
\end{aligned}
\end{equation}
and we show below that
\begin{equation}\label{V.zeta}
\sup_{x>x_{\la}+\frac12 x_{\la}^{-\nu}}\frac{\left|V\left(\frac{x}{x_\la}\right)\right|^\frac12}{|\zeta(x)|} = o(1), \qquad \la \to +\infty.
\end{equation}
To prove \eqref{V.zeta}, notice that for $x > x_\la$ and assuming that $\lambda$ is sufficiently large that $x_\lambda > \xi_0$
\begin{equation}
\left(
\frac{V(x)-\la}{V(x)}
\right)' = 
\frac{\la V'(x)}{V(x)^2} >0
\end{equation}
and, using \eqref{V.asm} and \eqref{V.Delta.new},
\begin{equation}
\frac{V(x_\la+\frac12 x_\la^{-\nu}) -V(x_\la)}{V(x_\la+\frac12 x_\la^{-\nu})} 
\approx 
\frac{V'(x_\la) x_\la^{-\nu}}{V(x_\la)} 
\approx 1.
\end{equation}
Thus, for $x>x_{\la}+x_{\la}^{-\nu}/2$,
\begin{equation}
\begin{aligned}
|\zeta(x)| &= \int_{x_\la}^x (V(t)-\la)^\frac 12 \, \dd t 
= \int_{x_\la}^x \frac{V'(t)}{V'(t)} (V(t)-\la)^\frac 12  \, \dd t 
\gs \frac{(V(x)-\la)^\frac32}{\max_{x_\la \leq t \leq x} V'(t)} 
\\
&= 
\frac{(V(x)-\la)^\frac32}{V(x)^\frac32} \frac{V(x)^\frac32}{\max_{x_\la \leq t \leq x} V'(t)} 
\gs
\min\{x_\la^{-\nu},x^{-\nu}\} V(x)^\frac12.
\end{aligned}
\end{equation}
Hence for $\nu<0$ we immediately arrive at
\begin{equation}
\frac{\left|V\left(\frac{x}{x_\la}\right)\right|^\frac12}{|\zeta(x)|} 
\ls 
\frac{\left|V\left(\frac{x}{x_\la}\right)\right|^\frac12}{V(x)^\frac 12 x_\la^{|\nu|}} \leq \frac{1}{x_\la^{|\nu|}}.
\end{equation}
For $\nu \geq 0$, we use \eqref{nu.xt.ineq.new} to get (with $\xi_0>0$ from Assumption~\ref{asm:V} and some $c>0$)
\begin{equation}
\begin{aligned}
\frac{\left|V\left(\frac{x}{x_\la}\right)\right|^\frac12}{|\zeta(x)|} 
&\ls 
\frac{x^{\nu} \left|V\left(\frac{x}{x_\la}\right)\right|^\frac12}{V(x)^\frac 12} 
\\
&\ls
\max_{x_\la \leq x \leq \xi_0 x_\la}\left(\frac{x^{2\nu}}{V(x)}\right)^\frac12
+ x^{\nu} \exp \left(
-c x^{\nu+1}(1+ \BigO(x_\la^{-\nu-1})) \right),
\end{aligned}
\end{equation}
thus \eqref{V.zeta} follows also in this case (recall \eqref{kappa.lem}).

As a consequence of \eqref{f.exp.1} and \eqref{V.zeta} we obtain in particular that
\begin{equation}\label{f.V.2}
\esssup_{x \geq x_\la+\frac12 x_\la^{-\nu}} \left|f \left(\frac x{x_{\la}} \right)\right| \exp(-|\zeta(x)|) = \BigO(1), \quad \la \to +\infty
\end{equation}
which we use in the estimate of integral
\begin{equation}
\cI_3:=\frac{\la^\frac12}{x_{\la}} \int_{x_{\la}+\frac12 x_\la^{-\nu}}^{\infty} y(x)^2 \left|f \left(\frac x{x_{\la}} \right)\right| \; \dd x. 
\end{equation}
In detail, employing \eqref{f.V.2}, \eqref{|u|.|v|.est}, \eqref{r.est}, changing the integration variables $-\ii \zeta(x) = |\zeta(x)|=t$ and using \eqref{V.Delta.new} and \eqref{V.asm} in the last steps, we get
\begin{equation}\label{I3.est}
\begin{aligned}
\cI_3& \ls 
\frac{\la^\frac12}{x_{\la}} \int_{x_{\la}+ \frac12 x_\la^{-\nu}}^{\infty} \frac{(1+C(\la)^2)e^{-| \zeta(x)|}}{(V(x) - \la)^\frac 12} \; \dd x
\\
&\ls 
\frac{\la^\frac12}{x_{\la}} \frac{1}{V(x_{\la}+\frac12 x_\la^{-\nu}) - V(x_{\la})} \int_0^\infty e^{-t} \; \dd t
\ls \frac{\la^\frac12}{x_{\la}} \frac{1}{V'(x_{\la}) x_\la^{-\nu}} 
\ls \frac{1}{x_{\la} \la^{\frac 12}}.
\end{aligned}
\end{equation}

Thus in summary, using \eqref{V.asm}, \eqref{kappa.lem} and $\nu \geq-1$, we get
\begin{equation}
\cI_1 + \cI_2 + \cI_3 \ls \left( \frac{x_\la^{3\nu}}{a_\la}\right)^\frac 16 \frac{1}{x_\la^{1+\nu}} +\frac{1}{x_{\la} \la^{\frac 12}} = o(1), \quad \la \to +\infty.
\end{equation}

We continue with the integral over $(0,x_\la-\delta)$, see \eqref{phik.f.int}, where we use the representation of $u^2$ from \eqref{u^2.exp}, \ie~
\begin{equation}\label{y2.dec}
y^2=  \frac{\pi}{(\la-V)^\frac12} (1 + \sin 2 \zeta + R_1(\zeta)) + 2 u r + r^2.
\end{equation}
The main contribution in \eqref{phik.f.int} reads (employing Lemma~\ref{lem:Om.beta})
\begin{equation}
\begin{aligned}
\cI_4&:=\frac{\la^\frac12}{x_{\la}} \int_{0}^{x_{\la}-\delta} \frac{\pi f \left(\frac{x}{x_{\la}}\right)}{(\la-V(x))^\frac12}  
 \; \dd x 
= 
\pi \int_0^{1-\frac{\delta}{x_{\la}}} \left(1-\frac{V(x_{\la} x)}{V(x_{\la})}\right)^{-\frac12}   
f(x) \; \dd x 
\\
&= 
\int_0^1 \frac{\pi f(x) \, \dd x}{(1-\omega_\beta(x))^\frac12}    
+ o(1), \quad \la \to +\infty.
\end{aligned}
\end{equation}
Thus, to prove \eqref{lim.nu.k}, we need to show that the remaining terms are negligible. 

Employing the estimates on $|u|$, $|r|$, see \eqref{|u|.|v|.est}, \eqref{r.est}, we get by changing the integration variables $x=x_{\la} t$ and applying Lemma~\ref{lem:Om.beta} that (recall that $f \in L^\infty_{\rm loc}(\R)$)
\begin{equation}
\begin{aligned}
\cI_5&:=\frac{\la^\frac12}{x_{\la}} \int_{0}^{x_{\la}-\delta} \left(|u(x)| |r(x)| + r(x)^2 \right)  \left|f \left(\frac x{x_{\la}} \right)\right| \; \dd x
\\
& \ls \frac{\la^\frac12}{x_{\la}} \int_{0}^{x_{\la}-\delta} \frac{C(\la)+C(\la)^2}{(\la-V(x))^\frac12} \; \dd x
\ls \Omega_\beta'  C(\la) = o(1), \quad \la \to +\infty. 
\end{aligned}
\end{equation}
Thus the contribution from the integrals with $2 u r+r^2$ is indeed negligible.

Using \eqref{R1.est}, \eqref{zeta.est2}, \eqref{V.xla.lem}, \eqref{V.asm} and \eqref{kappa.lem}, we obtain (recall that $f \in L^\infty_{\rm loc}(\R)$, $-\zeta' = (\la-V)^\frac12$ and see also \eqref{R1.est.lem})
\begin{equation}
\begin{aligned}
\cI_6&:=\frac{\la^\frac12}{x_{\la}} \int_{0}^{x_{\la}-\delta} \frac{|R_1(\zeta)|}{(\la - V(x))^\frac12} 
  \left|f \left(\frac x{x_{\la}} \right)\right|
\; \dd x
\\
&\ls
\frac{\la^\frac12}{x_{\la}} 
\left(
\frac{1}{\zeta(x_\la- \frac 12 x_\la^{-\nu})} \int_0^{x_\la-\frac 12 x_\la^{-\nu}} \frac{\dd x }{(\la-V(x))^\frac12}  
\right.
\\ & \qquad \qquad \left.
+ \int_{x_\la- \frac12 x_\la^{-\nu}}^{x_\la-\delta} \frac{\dd x }{\zeta(x)(\la-V(x))^\frac12}
\right)
\\
&\ls 
\frac{\la^\frac12}{x_{\la}} \left(
\left(\frac{x_\la^{3\nu}}{a_\la}\right)^\frac12 \int_0^{x_\la-\frac 12 x_\la^{-\nu}} \frac{\dd x }{(\la-V(x))^\frac12}  
+ \frac{\log \zeta(x_\la- \frac12 x_\la^{-\nu})}{\la-V(x_\la-\delta)}  
\right)
\\
& \ls \left(\frac{x_\la^{3\nu}}{a_\la}\right)^\frac12 \Omega_\beta'
+ \left( \frac{x_\la^{3\nu}}{a_\la}\right)^\frac 16 \frac{\log \frac{a_\la}{x_\la^{3\nu}} }{x_\la^{1+\nu}} = o(1), \quad \la \to +\infty. 
\end{aligned}
\end{equation}
Finally, we analyze the term with $\sin 2 \zeta$, see \eqref{y2.dec}. For every $\eps>0$ there is $g \in C_0^\infty((0,1))$ such that $\|f-g\|_{L^1((0,1))}<\eps$. With this $\eps>0$, we define $\delta_\eps:=\eps x_\la^{-\nu}$; notice that $\delta = o(\delta_\eps)$ as $\la \to +\infty$, see \eqref{delta.est}. Then
\begin{equation}\label{sin.dens}
\begin{aligned}
&\frac{\la^\frac12}{x_{\la}}
\left|
\int_0^{x_{\la}-\delta} \frac{\sin 2 \zeta(x)}{(\la-V(x))^\frac12} f\left(\frac{x}{x_{\la}}\right) \, \dd x
\right|
\\
& \quad \leq 
\frac{\la^\frac12}{x_{\la}}
\int_{x_\la-\delta_\eps}^{x_\la-\delta} \frac{1}{(\la-V(x))^\frac12} \left|f\left(\frac{x}{x_{\la}}\right) \right|  \, \dd x
\\
& \qquad
+ \la^\frac12
\int_0^{1-\frac{\delta_\eps}{x_\la}} \frac{|f(t)-g(t)|}{(\la-V(x_\la t))^\frac12}   \, \dd t
\\ & \qquad+
\frac{\la^\frac12}{x_{\la}}
\left|
\int_0^{x_{\la}-\delta_\eps} \frac{\sin 2 \zeta(x)}{(\la-V(x))^\frac12} g\left(\frac{x}{x_{\la}}\right) \, \dd x
\right|
\\
& \quad =: \cI_8 + \cI_9 + \cI_{10}.
\end{aligned}
\end{equation}
Using that $f \in L^\infty_{\rm loc}(\R)$, \eqref{V.Delta.new} and \eqref{V.asm}
\begin{equation}
\cI_8 \ls \frac{\la^\frac12}{x_{\la}} \frac{(V(x_\la)- V(x_\la-\delta_\eps))^\frac12}{V'(x_\la)}
\ls \eps^\frac12 \frac{\la^\frac12}{x_{\la}} \frac{(V'(x_\la) x_\la^{-\nu})^\frac12}{V'(x_\la)}
\ls  \frac{\eps^\frac12}{x_\la^{1+\nu}}.
\end{equation}
From $\|f-g\|_{L^1((0,1))}<\eps$, \eqref{V.Delta.new} and \eqref{V.asm}, we get
\begin{equation}
\cI_9 \ls \eps \frac{\la^\frac12}{(V(x_\la)-V(x_\la-\delta_\eps))^\frac12 }
\ls \eps \frac{\la^\frac12}{(V'(x_\la) \eps x_\la^{-\nu})^\frac12 } 
\ls\eps^\frac12.
\end{equation}
By integration by parts and \eqref{V.xla}, 
\begin{equation}\label{sin.est}
\begin{aligned}
\cI_{10}
&\ls 
\frac{\la^\frac12}{x_{\la}}
\left(
 \left|
\left[ g\left(\frac{x}{x_{\la}}\right) \frac{\cos 2 \zeta(x)}{\la-V(x)} \right]_0^{x_{\la}-\delta_\eps}
\right|
+
 \int_0^{x_{\la}-\delta_\eps} \left|\left(  \frac{g\left(\frac{x}{x_{\la}}\right)}{\la-V(x)}\right)' \right| \, \dd x
\right)
\\
& 
\ls
\frac{\la^\frac12}{x_{\la}} \left(
\frac{\|g\|_{L^\infty}}{\eps \la}  + \int_0^{x_{\la}-\delta_\eps} \frac{\|g'\|_{L^\infty}}{x_{\la}(\la -V(x))} + \frac{\|g\|_{L^\infty} V'(x)}{(\la -V(x))^2} \, \dd x
\right)
\\
& \ls \frac{\|g\|_{L^\infty}+\|g'\|_{L^\infty}}{\eps x_{\la} \la^\frac12}. 
\end{aligned}
\end{equation}

Putting the estimates from above together, we finally obtain
\begin{equation}
\limsup_{\la \to +\infty} \left|\int_0^\infty  \phi(x) f(x) \, \dd x -  \frac{1}{\Omega_\beta'} \int_0^1 \frac{f(x) \, \dd x}{(1-\omega_\beta(x))^\frac12} \right|
\ls \eps^\frac12,
\end{equation}
thus the claim \eqref{lim.nu.k} follows since $\eps>0$ was arbitrary.
\hfill \qed

\section{Eigenfunctions of Schr\"odinger operators with even single-well potentials}
\label{sec:EF.proofs}

In this section, we collect technical lemmas and proofs of results summarized in Section~\ref{subsec:EF.sum}; these are used in the proof of the main Theorem~\ref{thm:lim.EF}. Notice that in this section we \emph{do not assume} that \eqref{asm:V*} holds. The proofs follow mostly the reasoning in \cite[\S22.27]{Titchmarsh-1958-book2} and \cite{Giertz-1964-14}.

\begin{lemma}\label{lem:delta}
Let $V$ satisfy Assumption~\ref{asm:V}, let $\xi_0$ be as in \eqref{V.pos}, let $x_\la$, $a_\la$, $\zeta$ be as in \eqref{u.def} and $\delta$, $\delta_1$ as in \eqref{delta.def}. Let $\eps \in (0,1)$. Then, for all sufficiently large $\la >0$ and all sufficiently large $x$, the following hold.
\begin{align}
V^{(j)}(x + \Delta) &\approx V^{(j)}(x), \qquad |\Delta| \leq \eps x^{-\nu}, \quad j=0,1,
\label{V.Delta.lem}
\\
|\zeta(x_\la \pm \eps x_\la^{-\nu})| &\approx \left(\frac{a_\la}{x_\la^{3\nu}} \right)^\frac12,
\label{zeta.est}
\\
\label{delta.est.lem}
\delta &\approx \delta_1 \approx a_\la^{-\frac13},
\\
V(x_\la)-V(x_\la-\delta) &\approx a_\la \delta \approx a_\la^{\frac 23}, 
\quad
V(x_\la+\delta_1)-V(x_\la) \approx a_\la \delta_1 \approx a_\la^{\frac 23}.
\label{V.xla.lem}
\end{align}
\end{lemma}
\begin{proof}
	Using Assumption~\ref{asm:V}, for $\nu>-1$, we have
	\begin{equation}
	\begin{aligned}
	\left|\log \frac{V(x+\Delta)}{V(x)}\right| 
	&= 
	\left|\int_x^{x+\Delta} \frac{V'(t)}{V(t)} \, \dd t \right|
	\ls
	\left| |x+\Delta|^{\nu +1}-|x|^{\nu +1} \right|
	\\
	&\ls 
	x^{\nu} |\Delta|  + \BigO(|\Delta|^2 x^{\nu-1}),
	\end{aligned}
	\end{equation}
for $\nu=-1$,
\begin{equation}
\begin{aligned}
\left|\log \frac{V(x+\Delta)}{V(x)}\right| 
&= 
\left|\int_x^{x+\Delta} \frac{V'(t)}{V(t)} \, \dd t \right|
\ls
\left| \log \left(1 +  \frac{\Delta}{x} \right) \right|
\\
& \leq \max\{|\log(1-\eps)|, |\log(1+\eps)|\}; 
\end{aligned}
\end{equation}
the case with $j=1$ is similar.
	
Using \eqref{V.Delta.lem} for $V'$ and the mean value theorem in the last step, we get
	\begin{equation}\label{zeta.est2}
	\begin{aligned}
	\zeta(x_\la - \eps x_\la^{-\nu}) &= \int_{x_\la-\eps x_\la^{-\nu} }^{x_\la} \frac{V'(t)}{V'(t)}(\la- V(t))^\frac12 \, \dd t
	\\
	&\approx \frac{1}{a_\la} (V(x_\la)- V(x_\la-\eps x_\la^{-\nu}))^\frac32
	\approx \left(\frac{a_\la}{x_\la^{3\nu}} \right)^\frac12;
	\end{aligned}
	\end{equation}
	the case with $x_\la + \eps x_\la^{-\nu}$ is analogous.
	
	The number $\delta$ must satisfy 
	\begin{equation}\label{delta.oxnu}
	\delta = o(x_\la^{-\nu}), \qquad \la \to +\infty
	\end{equation}
	for otherwise $\zeta(x_\la-\delta ) \to + \infty$ by \eqref{zeta.est} and \eqref{kappa.lem}. Then, using the definition of $\delta$, see \eqref{delta.def}, we get similarly as in \eqref{zeta.est2},
	\begin{equation}\label{zeta.delta}
	1 = \zeta(x_\la - \delta) = \int_{x_\la-\delta }^{x_\la} \frac{V'(t)}{V'(t)}(\la- V(t))^\frac12 \, \dd t
	\approx \frac{1}{a_\la} (a_\la \delta)^\frac32
	\end{equation}
	and thus \eqref{delta.est.lem} follows. The reasoning for $\delta_1$ is analogous.
	
	Relations \eqref{V.xla.lem} follow by the mean value theorem, \eqref{delta.oxnu}, \eqref{V.Delta.lem} and \eqref{delta.est.lem}.
\end{proof}

\begin{lemma}\label{lem:w12}
	Let $V$ satisfy Assumption~\ref{asm:V}, let $u$, $v$ be as in \eqref{u.def} and let $w_1, w_2$ be as in \eqref{w12.def}. Then, for all sufficiently large $\la>0$, we have  
	
	\begin{equation}\label{|u|.|v|.est.lem}
	|u(x)| \ls \big(w_1(x) w_2(x) \big)^{-1}, \qquad |v(x)| \ls  w_1(x)^{-1} w_2(x),  \qquad x >0.
	\end{equation}	
\end{lemma}
\begin{proof}
	For $x \in (0,x_\la-\delta) \cup (x_\la+\delta_1, \infty)$, where $|\zeta|>1$, the inequalities \eqref{|u|.|v|.est.lem} follow from the definitions of $u$ and $v$ and asymptotic expansions of the corresponding Bessel functions for a large argument, see \eg~\cite[Chap.~10]{DLMF}; we omit details. 
	
	In the region around the turning point $x_\la$, one has $|\zeta| \leq 1$ and so expansions of Bessel functions for a small argument are used, see \eg~\cite[Chap.~10]{DLMF}. More precisely, for $u$ and $x_\la-\delta \leq x \leq x_\la$, one has, see \eqref{u.Bessel},
	\begin{equation}
	|u(x)| = \frac{\pi}{\sqrt 3} \left| b \right|
	\left|
	J_\frac 13(\zeta) + J_{-\frac 13}(\zeta)
	\right| \ls  \left(\frac{\zeta}{|\zeta'|^3} \right)^{\frac 16}. 
	\end{equation}
	Similarly as in \eqref{zeta.est2}, we obtain
	\begin{equation}
	\zeta(x) \approx \frac{(\la - V(x))^\frac32}{a_\la}  = \frac{|\zeta'(x)|^3}{a_\la}  ,  \qquad x_\la-\delta \leq x \leq x_\la,
	\end{equation}
	thus $|u(x)| \approx a_\la^{-\frac 16}$. The case $x_\la < x < x_\la + \delta_1$ is similar.
	
	The estimates for $v$ are obtained analogously.
\end{proof}

\begin{lemma}\label{lem:u^2}
	Let $V$ satisfy Assumption~\ref{asm:V} and $u$, $x_\la$ and $a_\la$ be as in \eqref{u.def}. Then
	\begin{equation}\label{u.norm.lem}
	\begin{aligned}
	\int_0^{\infty} u(x)^2 \, \dd x 
	&= 
	\left(\int_0^{x_\la} \frac{\pi \, \dd x}{(\la - V(x))^\frac12}\right) 
	\left(
	1 + \BigO
	\left(\frac{1}{x_\la}+
	\left(\frac{x_\la^{3\nu}}{a_\la}\right)^\frac16   \frac{\log \frac{a_\la}{x_\la^{3\nu}}}{x_\la^{1+\nu}}  
	\right)
	\right)
	\\
	&= \left(\int_0^{x_\la} \frac{\pi \, \dd x}{(\la - V(x))^\frac12} \right) (1+o(1)), \quad \la \to +\infty.
	\end{aligned}
	\end{equation}	
\end{lemma}
\begin{proof}
	Using \eqref{u^2.exp}, we obtain
	\begin{equation}
	\begin{aligned}
	\int_0^{\infty} u(x)^2 \, \dd x &=
	\int_0^{x_\la} \frac{\pi}{(\la-V(x))^\frac12} \, \dd x +  
	\pi \int_0^{x_\la-\delta} \frac{\sin 2 \zeta(x) + R_1(\zeta(x))}{(\la-V(x))^\frac12} \, \dd x
	\\
	& \quad + \int_{x_\la-\delta}^{x_\la+\delta_1} u(x)^2 \, \dd x 
	+ \int_{x_\la+\delta_1}^\infty u^2(x) \, \dd x 
	\\
	& \quad
	- \int_{x_\la-\delta}^{x_\la} \frac{\pi}{(\la-V(x))^\frac12} \, \dd x.
	\end{aligned}
	\end{equation}
	First we notice that
	\begin{equation}
	\int_0^{x_\la} \frac{\dd x}{(\la-V(x))^\frac12} 
	= \frac{1}{\la^\frac12}\int_0^{x_\la} \frac{\dd x}{(1-\frac{V(x)}\la)^\frac12} 
	\gs \frac{x_\la}{\la^\frac12}.
	\end{equation}
	Using \eqref{|u|.|v|.est.lem} and \eqref{delta.est.lem}, we get
	\begin{equation}
	\int_{x_\la-\delta}^{x_\la+\delta_1} u(x)^2 \, \dd x \ls a_\la^{-\frac 23}.
	\end{equation}
	Since $\delta \approx a_\la^{-\frac13} = o(x_\la^{-\nu})$ as $\la \to +\infty$, see \eqref{delta.est.lem} and \eqref{delta.oxnu}, using \eqref{V.Delta.lem}, we get
	\begin{equation}
	\int_{x_\la-\delta}^{x_\la} \frac{\dd x}{(\la-V(x))^\frac12} 
	=
	\int_{x_\la-\delta}^{x_\la} \frac{V'(x) \, \dd x}{V'(x)(\la-V(x))^\frac12} 
	\ls 
	\frac{(a_\la \delta)^\frac12}{a_\la} \approx a_\la^{-\frac 23}.
	\end{equation}
	Using \eqref{|u|.|v|.est.lem}, the definition \eqref{delta.def} of $\delta_1$ and \eqref{V.xla.lem}, we have
	\begin{equation}
	\begin{aligned}
	\int_{x_\la+\delta_1}^{\infty} u(x)^2 \, \dd x 
	&\ls
	\int_{x_\la+\delta_1}^{\infty} \frac{e^{-2 \int_{x_\la}^x (V(s)-\la)^\frac12 \, \dd s}}{(V(x)-\la)^\frac12} \, \dd x
	\\
	&
	\ls 
	\frac{1}{V(x_\la+\delta_1 ) - \la} \int_1^\infty e^{-2 t} \, \dd t 
	\ls \frac{1}{a_\la \delta_1 } \approx a_\la^{-\frac 23}.
	\end{aligned}
	\end{equation}
	The second mean value theorem for integrals (from which the point $\xi_1=\xi_1(\la)$ arises below), the fact that $V$ is increasing for $x>\xi_0$ (see \eqref{V.pos}) and \eqref{V.xla.lem} yield (recall that by \eqref{u.def} $-\zeta' = (\la-V)^\frac12$)
	\begin{equation}
	\begin{aligned}
	\left|\int_0^{x_\la-\delta} \frac{\sin 2 \zeta(x) \, \dd x}{(\la-V(x))^\frac12} \right|
	& \ls  \la^{-\frac12} +
	\frac{1}{\la-V(\xi_0)} \left|\int_{\xi_0}^{\xi_1} (- \zeta'(x)) \sin 2 \zeta(x) \, \dd x\right|
	\\ &\quad + 
	\frac{1}{\la-V(x_\la-\delta)} \left| \int_{\xi_1}^{x_\la-\delta} (- \zeta'(x)) \sin 2 \zeta(x) \, \dd x\right|  
	\\
	&\ls \la^{-\frac12} +
	 \frac{1}{a_\la^{\frac 23}} \left| \int_{1}^{\zeta(\xi_1)} \sin 2 t \, \dd t \right| 
	\ls   \la^{-\frac12} + a_\la^{-\frac 23}.
	\end{aligned}
	\end{equation}
	Using \eqref{R1.est}, \eqref{zeta.est2} and \eqref{V.xla.lem}, we have 
	\begin{equation}\label{R1.est.lem}
	\begin{aligned}
	& \int_0^{x_\la-\delta} \frac{|R_1(\zeta(x))|}{(\la-V(x))^\frac12} \, \dd x 
	\\
	& \quad \ls
	\int_0^{x_\la-\delta} \frac{\dd x }{\zeta(x)(\la-V(x))^\frac12} 
	\\
	&\quad \ls
	\frac{1}{\zeta(x_\la- \frac12 {x_\la^{-\nu}})} \int_0^{x_\la-\delta} \frac{\dd x }{(\la-V(x))^\frac12}  
	+ \int_{x_\la- \frac12 {x_\la^{-\nu}}}^{x_\la-\delta} \frac{\dd x }{\zeta(x)(\la-V(x))^\frac12}
	\\ 
	&\quad \ls \left(\frac{x_\la^{3\nu}}{a_\la}\right)^\frac12 \int_0^{x_\la-\delta} \frac{\dd x }{(\la-V(x))^\frac12}  
	+ \frac{\log \zeta(x_\la- \frac12 {x_\la^{-\nu}})}{V(x_\la)-V(x_\la-\delta)}  
	\\
	& \quad \ls \left(\frac{x_\la^{3\nu}}{a_\la}\right)^\frac12 \int_0^{x_\la-\delta} \frac{\dd x }{(\la-V(x))^\frac12}  
	+ \frac{\log \frac{a_\la}{x_\la^{3\nu}} }{ a_\la^{\frac 23}}. 
	\end{aligned}
	\end{equation}
	From \eqref{V.asm} we have
	\begin{equation}
	\frac{\la^\frac12}{x_\la a_\la^\frac23} \approx \left( \frac{x_\la^{3\nu}}{a_\la}\right)^\frac16 \frac{1}{x_\la^{1+\nu}},
	\end{equation}
	thus the claim \eqref{u.norm.lem} follows by putting together  all estimates from above (and \eqref{kappa.lem}).
\end{proof}

\begin{lemma}\label{lem:J}
	Let $V$ satisfy Assumption~\ref{asm:V}, let $K$ be as in \eqref{K.def}, let $w_1$ be as in \eqref{w12.def} and let $\kappa_\la$ be as in \eqref{kappa.def}. Then
	\begin{equation}\label{J.def}
	\cJ_K(\la):= \int_0^\infty \frac{K(s)}{w_1(s)^2} \, \dd s = \BigO(\la^{-\frac12} + \kappa_\la) = o(1), \qquad \la \to +\infty.
	\end{equation} 
\end{lemma}
\begin{proof}
	We follow and extend the strategy in \cite[\S22.27]{Titchmarsh-1958-book2}. We split the integral into several regions; we define $\delta_\la':=\eps_1 x_\la^{-\nu}$ and $\delta_\la'':=\eps_2 x_\la^{-\nu}$, where $\eps_1, \eps_2 \in (0,1)$ will be determined below.
	
	\noindent
	$\bullet$ $0 \leq s \leq \xi_0$: Notice that $\zeta(s) \gs \la^\frac12$, hence (recall that $-\zeta' = (\la-V)^\frac12$)
	\begin{equation}
	\int_0^{\xi_0}  \frac{|K(s)|}{w_1(s)^2} \, \dd s 
	\ls
	\int_0^{\xi_0} \frac{-\zeta'(s)}{\zeta(s)^2} \, \dd s + \frac{1}{\la^\frac12} \ls \frac{1}{\la^\frac12}.
	\end{equation}
	
	\noindent
	$\bullet$ $\xi_0 \leq s \leq x_\la-\delta_\la'$: We give the estimate for any value of $\eps_1 \in (0,1)$; $\eps_1$ will be specified below, see \eqref{eps1.choice}, 
	\begin{equation}\label{reg.2.est}
	\begin{aligned}
	\int_{\xi_0}^{x_\la-\delta_\la'}  \frac{|K(s)|}{w_1(s)^2} \, \dd s 
	&\ls
	\int_{\xi_0}^{x_\la-\delta_\la'} \frac{-\zeta'(s)}{\zeta(s)^2} \, \dd s 
	+ \left|\int_{\xi_0}^{x_\la-\delta_\la'} \frac{V''(s) \, \dd s}{(\la-V(s))^\frac32} \right| 
	\\
	& \quad 
	+ \int_{\xi_0}^{x_\la-\delta_\la'} \frac{V'(s)^2 \, \dd s}{(\la-V(s))^\frac52}.  
	\end{aligned}
	\end{equation}
	The first integral on the r.h.s. is estimated using \eqref{zeta.est2} 
	\begin{equation}
	\int_{\xi_0}^{x_\la-\delta_\la'} \frac{-\zeta'(s)}{\zeta(s)^2} \, \dd s \leq \frac{1}{\zeta(x_\la-\delta_\la')} 
	\ls 
	\left( \frac{x_\la^{3\nu}}{a_\la}\right)^\frac12.
	\end{equation}
	Since by \eqref{V.asm}
	\begin{equation}
	\la-V(x_\la-\delta_\la') \approx a_\la \delta_\la' \approx \la, 
	\end{equation}
	we have for the third integral on the r.h.s. in \eqref{reg.2.est} that (we use \eqref{V.asm} and \eqref{kappa.lem})
	\begin{equation}
	\begin{aligned}
	\int_{\xi_0}^{x_\la-\delta_\la'} \frac{V'(s)^2 \, \dd s}{(\la-V(s))^\frac52} 
	&\ls
	\frac{\la \max\{1,x_\la^\nu\}}{\la^\frac52} \int_{\xi_0}^{x_\la-\delta_\la'} V'(s) \, \dd s 
	\\
	&\ls \max \left\{\frac{1}{\la^\frac12},\left( \frac{x_\la^{3\nu}}{a_\la}\right)^\frac12  \right\}.
	\end{aligned}
	\end{equation}
	Integration by parts in the second integral on the r.h.s. in \eqref{reg.2.est}, the choice of $\delta_\la'$ and \eqref{V.Delta.lem} lead to
	\begin{equation}
	\begin{aligned}
	\left|\int_{\xi_0}^{x_\la-\delta_\la'} \frac{V''(s) \, \dd s}{(\la-V(s))^\frac32} \right|  
	&\ls
	\frac{V'(x_\la-\delta_\la')}{(\la-V(x_\la-\delta_\la'))^\frac32} + \int_{\xi_0}^{x_\la-\delta_\la'} \frac{V'(s)^2 \, \dd s}{(\la-V(s))^\frac52}
	\\
	&\ls \max \left\{\frac{1}{\la^\frac12},\left( \frac{x_\la^{3\nu}}{a_\la}\right)^\frac12  \right\}.
	\end{aligned}
	\end{equation}
	Putting together the estimates above, we arrive at
	\begin{equation}
	\int_{0}^{x_\la-\delta_\la'}  \frac{|K(s)|}{w_1(s)^2} \, \dd s \ls  \frac{1}{\la^\frac12} +\left( \frac{x_\la^{3\nu}}{a_\la}\right)^\frac12.
	\end{equation}

	\noindent
	$\bullet$ $x_\la+\delta_\la'' \leq s$: The estimates are again obtained for any value of $\eps_2 \in (0,1)$ which will be specified later. The important observations are (based on the choice of $\delta_\la''$ and \eqref{V.asm})
	\begin{equation}
	\begin{aligned}
	V(x_\la+\delta_\la'')-V(x_\la) &\approx a_\la x_\la^{-\nu} \approx \la,
	\\
	|\zeta(x_\la+\delta_\la'')| &\gs \left( \frac{a_\la}{x_\la^{3\nu}}\right)^\frac12.
	\end{aligned}
	\end{equation}
	Moreover, since $V'(x)>0$ for all sufficiently large $x>0$, 
	\begin{equation}
	\left(\frac{V(x)}{V(x)-\la}\right)' = -\frac{\la V'(x)}{(V(x)-\la)^2} <0,
	\end{equation}
and (see \eqref{V.asm})
	\begin{equation}
	\frac{V(x_\la+ \delta_\la'')}{V(x_\la+ \delta_\la'')-V(x_\la)} \approx \frac{\la}{a_\la x_\la^{-\nu}} \approx 1, 	
	\end{equation}
	we obtain (recall \eqref{kappa.lem})
	\begin{equation}\label{reg.3.est}
	\begin{aligned}
	\int_{x_\la+\delta_\la''}^\infty  \frac{|K(s)|}{w_1(s)^2} \, \dd s 
	&\ls
	\int_{x_\la+\delta_\la''}^\infty \frac{|\zeta(s)|'}{|\zeta(s)|^2} \, \dd s 
	+ \int_{x_\la+\delta_\la''}^\infty \frac{|V''(s)| }{V(s)^\frac32} + 
	\frac{V'(s)^2 }{V(s)^\frac52} \, \dd s
	\\
	&\ls  \left( \frac{x_\la^{3\nu}}{a_\la}\right)^\frac12 + \kappa_\la \ls \kappa_\la.
	\end{aligned}
	\end{equation}
	
	\noindent
	$\bullet$ $x_\la -\delta_\la' \leq s \leq x_\la$: We integrate by parts twice in the formula for $\zeta$ and obtain
	\begin{equation}
	\zeta = \frac 23 \frac{(\la-V)^\frac32}{V'} \left(
	1- \frac25 \frac{(\la-V)V''}{V'^2} - T
	\right),
	\end{equation}
	where
	\begin{equation}
	T(s) = \frac25 \frac{V'(s)}{(\la-V(s))^\frac32} \int_s^{x_\la} (\la-V(t))^\frac52 \left(\frac{V''(t)}{V'(t)^3}\right)' \, \dd t.
	\end{equation}
	Using \eqref{V.asm}, we obtain
	\begin{equation}
	\frac{(\la-V(s))V''(s)}{V'(s)^2} \ls \frac{a_\la \delta_\la' x_\la^\nu}{a_\la} \ls \eps_1.
	\end{equation}
	To estimate $T$, we first notice that by \eqref{V.asm}, \eqref{V.Delta.lem} and \eqref{kappa.lem}
	\begin{equation}
	\left|\left(\frac{V''(t)}{V'(t)^3}\right)' \right| \ls \frac{|V'''(t)|}{V'(t)^3} + \frac{V''(t)^2}{V'(t)^4} 
	\ls \left( \frac{x_\la^\nu}{a_\la}\right)^2.
	\end{equation}
	Thus
	\begin{equation}\label{S.est}
	\begin{aligned}
	|T(s)| &\ls \frac{x_\la^{2\nu}}{a_\la^2(\la-V(s))^\frac32 } \int_s^{x_\la} V'(t) (\la-V(t))^\frac52 \, \dd t
	\ls \frac{x_\la^{2\nu}}{a_\la^2}(\la-V(s))^2
	\\
	&
	\ls \frac{x_\la^{2\nu}}{a_\la^2} (\la-V(x_\la-\delta_\la'))^2 
	\ls \eps_1^2.
	\end{aligned}
	\end{equation}
	Hence it is possible to select $\eps_1 \in (0,1)$ so small that 
	\begin{equation}\label{eps1.choice}
	\left|
	\frac25 \frac{(\la-V)V''}{V'^2} - T
	\right| \leq \frac 14
	\end{equation}
	and so, using Taylor's theorem for $\zeta^{-2}$ and cancellations in $K$, one arrives at (using \eqref{S.est},  \eqref{V.asm} and \eqref{w12.def})
	\begin{equation}
	\begin{aligned}
	\frac{|K(s)|}{w_1(s)^2} &\ls \frac{|K(s)|}{(\la-V(s))^\frac12}
	\ls
	\frac{V'(s)^2}{(\la-V(s))^\frac52} \left[ \left(\frac{(\la-V(s))V''(s)}{V'(s)^2}\right)^2 +  |T(s)| \right]
	\\
	&\ls \frac{x_\la^{2\nu}}{(\la-V(s))^\frac12}.
	\end{aligned}
	\end{equation}
	Hence,
	\begin{equation}
	\int_{x_\la-\delta_\la'}^{x_\la}  \frac{|K(s)|}{w_1(s)^2} \, \dd s 
	\ls \frac{x_\la^{2\nu}}{a_\la} (\la-V(x_\la-\delta_\la'))^\frac 12 
	\ls \left(\frac{x_\la^{3\nu}}{a_\la}\right)^\frac12.
	\end{equation}
	\noindent
	$\bullet$ $x_\la \leq s \leq x_\la + \delta_\la''$: The estimate and the choice of $\eps_2$ in this region is analogous to the previous case. We omit the details.
	
	In summary, putting all estimates together and using \eqref{kappa.lem}, we obtain the claim \eqref{J.def}. 
\end{proof}

\begin{proof}[{Proof of Theorem~\ref{thm:EF.asym}}]
	We follow the steps in \cite{Giertz-1964-14}; the main differences are the additional perturbation $W$ and new estimate of $\cJ(\la)$ from Lemma~\ref{lem:J}. 
	
	Using \eqref{uv.Wron} and variation of constants, we can find a solution (distributional, since $W \in L_{\rm loc}^1(\R)$ only) of \eqref{y.ode} by solving the integral equation
	\begin{equation}\label{y.int.eq}
	y(x) = u(x) + \int_x^\infty G(x,s) (K(s) + W(s)) y(s) \, \dd s,
	\end{equation}
	where $G(x,s) = u(x) v(s) - v(x) u(s)$. Using the notation $\hat f$ for a function $f$ multiplied by $w_1 w_2$, we rewrite the integral equation \eqref{y.int.eq} as
	\begin{equation}\label{y1.eq}
	\hat y(x) = \hat u(x) + \int_x^\infty H(x,s) \frac{K(s)+W(s)}{w_1(s)^2} \hat y(s) \, \dd s;
	\end{equation}
	here 
	\begin{equation}
	H(x,s) = \left( \hat u(x) \hat v(s) - \hat v(x) \hat u(s) \right) w_2(s)^{-2}
	\end{equation}
	and $|H(x,s)| \ls 1$ in $0 \leq x \leq s$, see \eqref{|u|.|v|.est}. 
	Let 
\begin{equation}\label{J.la.def}
\cJ_{K+W}(\la) := \int_0^\infty \frac{K(s)+W(S)}{w_1(s)^2} \, \dd s = \cJ_K + \cJ_W. 
\end{equation}
If $\cJ_{K+W}(\la) = o(1)$ as $\la \to +\infty$, then we can solve the equation \eqref{y1.eq} in $L^\infty(\R_+)$ and we can write the solution as
\begin{equation}\label{r1.est}
\hat y = \hat u + \hat r, \qquad \|\hat r\|_{L^\infty(\R_+)} \ls  \frac{\cJ_{K+W}(\la)}{1- \cJ_{K+W}(\la)}=:C(\la).
\end{equation}
Returning back to $y$, we obtain \eqref{yu.rel.ap} and \eqref{r.est}. 
	
The estimate on $\cJ_K$ is the main technical step of the proof, see Lemma~\ref{lem:J} above, the decay of $\cJ_W$ is guaranteed by Assumption~\ref{asm:W}.

Finally, the formula \eqref{y.norm} for the $L^2$-norm of $y$ follows from \eqref{u.norm} as in \cite[Thm.~1]{Giertz-1964-14}. Namely, 
	\begin{equation}
	y^2 = u^2 + \frac{\hat r(2 \hat u +\hat r)}{w_1^2 w_2^2}
	\end{equation}
	and
	\begin{equation}\label{w12.int.est}
	\begin{aligned}
	\int_0^\infty \frac{\dd x}{w_1(x)^2 w_2(x)^2} &=
	\int_0^{x_\la-\delta} \frac{\pi \, \dd x}{(\la-V(x))^\frac12} + \BigO\left(\frac{\delta+\delta_1}{a_\la^\frac13}\right) 
	\\
	& \quad + \int_1^\infty \frac{e^{-2t} \, \dd t}{(V(\zeta^{-1}(t))-\la)}
	\\
	&= 
	\int_0^{x_\la} \frac{\pi \, \dd x}{(\la-V(x))^\frac12} + \BigO(a_\la^{-\frac 23}), \qquad \la \to + \infty,
	\end{aligned}
	\end{equation}
	see the proof of Lemma~\ref{lem:u^2} for more details on the estimates. The claim \eqref{y.norm} then follows from \eqref{u.norm}, \eqref{w12.int.est} and $\|\hat r(2 \hat u +\hat r)\|_{L^\infty} \ls C(\la)$, see \eqref{r1.est} and \eqref{|u|.|v|.est}.
\end{proof}

\section{Comparison with existing results}
\label{sec:discussion}

\subsection{Concentration measures for orthogonal polynomials}

It is interesting to compare the concentration phenomenon \eqref{lim.nu.k} of measures \eqref{nu.k.def} with its analogue in the case of orthogonal polynomials $\{p_n(x)\}$ for the weights $\exp(-|x|^\alpha)$, $\alpha>0$, or even more general non-even weights $w(x)= \exp(-\tilde w(x))$ with properly chosen $\tilde w$. Following \cite{Kriecherbauer-1999-6,Levin-2001}, let 
\begin{equation}
\kappa_\alpha := \frac{\Gamma \left(\frac{\alpha}{2}\right)\Gamma \left(\frac{1}{2}\right)}{\Gamma \left(\frac{\alpha+1}{2}\right)},
\qquad 
w_\alpha(x):=\exp(-\kappa_\alpha|x|^\alpha ), \qquad \alpha >0;
\end{equation}
the corresponding system of orthogonal polynomials $\{p_n(x)\}$
\begin{equation}
\int_\R p_n(x) p_m(x) w_\alpha(x) \, \dd x = \delta_{mn}, \qquad m,n \in \Z,
\end{equation}
has the property, as $n \to \infty$, 
\begin{equation}\label{pn.1}
\begin{aligned}
&p_n(n^\frac 1\alpha x) \sqrt{w_\alpha(n^\frac 1\alpha x)} =
\\
&
\qquad 
\sqrt{\frac{2}{\pi n^{\frac 1 \alpha}}} (1-x^2)^{-\frac 14}
\left[
\cos \left(
n \pi \int_1^x \psi_\alpha(y) \; \dd y + \frac 12 \arcsin x
\right)
+ \BigO(n^{-1}
\right],
\end{aligned}
\end{equation}
where $0<\delta \leq x \leq 1 - \delta$ with $\delta$ arbitrarily small and 
\begin{equation}
\psi_\alpha(y) = \frac{\alpha}{\pi} x^{\alpha-1} \int_1^{\frac 1x} \frac{u^{\alpha-1}}{\sqrt{u^2-1}} \, \dd u.
\end{equation}
Formula \eqref{pn.1} and elementary trigonometry imply that, as $n \to \infty$,
\begin{equation}
\begin{aligned}
&n^\frac 1\alpha p_n^2(n^\frac 1\alpha x) w_\alpha(n^\frac 1\alpha x)
= 
\\
&
\qquad
\frac 1 \pi \frac{1}{\sqrt{1-x^2}}
\left[
1 + \frac12 \sin 
\left(
2 n \pi \int_1^x \psi_\alpha(y) \; \dd y + \frac 12 \arcsin x
\right)
+ \BigO(n^{-1})
\right].
\end{aligned}
\end{equation}
Thus, for any $f \in C([-1,1])$, Riemann-Lebesgue lemma gives 
\begin{equation}
\lim_{n \to \infty} \int_\delta^{1-\delta} f(x) n^\frac 1\alpha p_n^2(n^\frac 1\alpha x) w_\alpha(n^\frac 1\alpha x) \, \dd x = \frac{1}{\pi}\int_\delta^{1-\delta} \frac{f(x)}{\sqrt{1-x^2}} \; \dd x.
\end{equation}
Moreover, by \cite[Thm.1.16]{Kriecherbauer-1999-6},
\begin{equation}
\sup_{n \geq 1} \sup_{x \in \R} \, n^\frac 1\alpha p_n^2(n^\frac 1\alpha x) w_\alpha(n^\frac 1\alpha x) \sqrt{|1-x^2|} < \infty, 
\end{equation}
so
\begin{equation}
\lim_{n \to \infty} \int_{-1}^{1} f(x) n^\frac 1\alpha p_n^2(n^\frac 1\alpha x) w_\alpha(n^\frac 1\alpha x) \, \dd x = \frac{1}{\pi}\int_{-1}^{1} \frac{f(x)}{\sqrt{1-x^2}} \; \dd x.
\end{equation}

On the whole real line, one can use the following inequalities, see \cite[Thm.19, p.16, Eq.(1.66)]{Levin-2001}. Let $a>1$ and $P$ be a polynomial of degree smaller than or equal to $n$. Then 
\begin{equation}
\int_{|x| \geq a}  P^2(n^\frac 1 \alpha x) \omega_\alpha (n^\frac 1 \alpha x) \, \dd x 
\leq 
C_1  \exp(-C_2 n) \int_{-1}^1 P^2(n^\frac 1 \alpha x) \omega_\alpha (n^\frac 1 \alpha x) \, \dd x
\end{equation}
for all $n \geq 1$; the constants $C_1$, $C_2$ depend on $a$, but not on $n$ or $P$. These inequalities imply 
\begin{equation}\label{lim.OP}
\lim_{n \to \infty} \int_{- \infty}^\infty f(x) n^\frac 1\alpha p_n^2(n^\frac 1\alpha x) w_\alpha(n^\frac 1\alpha x) \, \dd x 
= 
\frac{1}{\pi}\int_{-1}^{1} \frac{f(x)}{\sqrt{1-x^2}} \; \dd x
\end{equation}
for any bounded continuous function on $\R$. 

A striking difference between \eqref{lim.OP} and \eqref{lim.nu.k} is that in the case of orthogonal polynomials the concentration measure \emph{does not depend on $\alpha$}, or $\tilde w$ in a more general case of weights $\exp(-\tilde w(x))$.    

\subsection{Semi-classical defect measures}
\label{subsec:s-c}

In classical mechanics, \cf~\cite{Arnold-1989-60}, a particle with position $x(t)$ subject to the differential equation
\begin{equation}\label{scm:ham_flow}
	\begin{cases}
	\ddot{x}(t) +V(x(t)) = 0,
	\\ 
	(x(0), \dot{x}(0)) = (x_0, \xi_0) 
	\end{cases}
\end{equation}
remains for all times on the energy surface
\[
	(x(t), \dot{x}(t)) \in \{(x,\xi) \::\: \xi^2 + V(x) = \xi_0^2 + V(x_0)\}
\]
and travels along the trajectory $(\dot{x}(t), \dot{\xi}(t))$ obeying
\[
	(\dot{x}(t), \dot{\xi}(t)) = (2\xi(t), -V'(x(t))).
\]
The classical-quantum correspondence suggests that, in the high-energy limit, the $L^2$-mass of an eigenfunction should be distributed in the same way as the average position of a classical particle: since a classical particle passes through an interval $[x_*, x_* + \dd x]$ in physical space with velocity near $\w(x_*)$ or $-\w(x_*)$, where 
\begin{equation}\label{scm.classicalvelocity}
	\w(x_*) = \sqrt{\lambda - V(x_*)},
\end{equation}
we obtain the heuristic (for a normalization constant $c_0$)
\begin{equation}\label{scm.classicalheuristic}
	|u(x)|^2\,\dd x \ ``=" \ \frac{c_0}{\w(x)}\,\dd x = \frac{c_0}{\sqrt{\lambda - V(x)}}\,\dd x,
\end{equation}
which agrees with Theorem \ref{thm:lim.EF} after the corresponding scaling.

To make this correspondence precise, one can use the notion of semiclassical defect measures (see, for instance, \cite[Ch.~5]{Zworski-2012}). The following discussion will be under weaker hypotheses than Theorem \ref{thm:lim.EF}, because our goal is only to show that the precise asymptotics obtained agree with the semiclassical prediction.

Let $V:\R \to \R$ be even, smooth and suppose that
there exists some $\beta > 0$ such that
\begin{equation}\label{scm.Vsymbol}
	\left|V^{(k)}(x)\right| \ls (1+|x|)^{\beta - k}, \qquad k \in \N_0, \quad x \in \R.
\end{equation}
Suppose also that
\begin{equation}\label{scm.Vwell}
\begin{aligned}
V'(x) &> 0, \quad x > 0. 
\end{aligned}
\end{equation}
and that there exists $x_0 > 0$ such that 
\begin{equation}\label{scm.Vgrowth}
V'(x) \gs (1+x)^{\beta-1}, \quad x > x_0;
\end{equation}
the latter implies that, for $|x|$ sufficiently large,
\[
	V(x) \approx (1+|x|)^{\beta}.
\]

We consider the semiclassical Schr\"odinger operator
\[
	A_{\hbar} = -\hbar^2\frac{\dd^2}{\dd x^2} + V(x)
\]
in the limit $\hbar \to 0^+$. 

For instance, if $V(x) = |x|^\beta$ for $\beta \in 2\N$, scaling gives a unitary equivalence
\[
	-\frac{\dd^2}{\dd x^2} + |x|^\beta \sim \hbar^{-\frac{2\beta}{2+\beta}}\left(-\hbar^2 \frac{\dd^2}{\dd x^2} + |x|^\beta\right).
\]
Other potentials can be treated by rescaling and controlling the error, but this analysis is outside the aim of this work. We emphasize that the assumptions on $Q$ in Theorem \ref{thm:lim.EF} are significantly weaker than the hypotheses on $V$ here, \cf~\eqref{asm:V*}, Assumption~\ref{asm:V} and \ref{asm:W} and comments in Introduction.

Suppose that for $\lambda_0 > \inf V(x)$, there exists a sequence $\{\hbar_k\}_{k \in \N}$ of positive numbers tending to zero and eigenfunctions $\{u_{k}\}_{k \in \N}$ obeying $\|u_k\| = 1$ and
\[
	A_{\hbar_k}u_{k} = \lambda_0 u_{k}.
\]

For each $u_{k}$, one can define the functional
\[
	\varphi_{k}(b) = \int_{\R} \overline{u_{k}(x)}b_{\hbar_k}^w(x,\hbar_k D_x)u_{k}(x)\,\dd x, \quad b \in C_c^\infty(\R).
\]
Here, $D_x=-\ii \frac{\dd}{\dd x}$ and $b_\hbar^w(x,\hbar D_x)$ is the Weyl quantization (see e.g.\ \cite[Ch.~4]{Zworski-2012}); when $b\in C_c^\infty(\R)$, the Weyl quantization of $b$ is a compact operator on $L^2(\R)$ which takes $\mathscr{S}'(\R)$ to $\mathscr{S}(\R)$.

Following \cite[Thm.~5.2]{Zworski-2012} there is a subsequence $\{u_{k_j}\}_{j \in \N}$ with $\hbar_{k_j} \to 0^+$ for which the functionals $\varphi_{k}$ converge to a non-negative Radon measure $\mu$ in the sense that, for each $b\in C_c^\infty(\R)$,
\begin{equation}\label{scm.mu.limit}
	\lim_{j\to\infty} \varphi_{k_j}(b) = \int_{\R^2} b(x,\xi)\,\dd \mu(x,\xi).
\end{equation}
We will show that this $\mu$ is unique and that therefore $\varphi_{k} \to \mu$ in the same sense since every subsequence admits a further subsequence tending to $\mu$.

By \cite[Thm.~5.3~or~Thm.~6.4]{Zworski-2012}, 
\begin{equation}\label{scm.suppmu}
	\operatorname{supp}\mu \subseteq \{\xi^2 + V(x) = \lambda_0\},
\end{equation}
so let us define, in analogy with \eqref{scm.classicalvelocity},
\begin{equation}\label{rho.def}
	\w(x) = \sqrt{\lambda_0 - V(x)}
\end{equation}
for those $x$ such that $V(x) < \lambda_0$. There exists a measure $\nu_+$ such that, when $\supp b \subset \{\xi > 0\}$, then
\begin{equation}\label{scm.nupush}
	\int_{\R^2} b(x,\xi)\,\dd \mu(x,\xi) = \int_{\{V(x) < \lambda_0\}}b(x, \w(x))\,\dd \nu_+(x).
\end{equation}

By \cite[Thm.~5.4]{Zworski-2012}, for any $b \in C_c^\infty(\R^2)$,
\begin{equation}\label{scm.bracket}
	\int_{\R^2} \{a, b\}(x,\xi)\,\dd \mu(x,\xi) = 0,
\end{equation}
where the Poisson bracket $\{a, b\}$ of the symbol $a(x,\xi) = \xi^2 + V(x)$ of $A_\hbar$ with $b$ is
\[
	\{a, b\} = a_\xi b_x - a_x b_\xi = 2\xi b_x - V'(x)b_\xi.
\]
This corresponds to invariance of $\mu$ under the classical Hamilton flow associated to $a(x,\xi)$, which in the case of a Schr\"odinger operator corresponds to \eqref{scm:ham_flow}.

Finally, since in our situation the support of $\mu$ is compact, we show that
\begin{equation}\label{scm.mass1}
	\int_{\R^2} \dd \mu(x,\xi) = 1
\end{equation}
as follows. For any $b(x,\xi) \in C_c^\infty(\R)$ such that $b \equiv 1$ on $\{\xi^2 + V(x) = \lambda_0\}$, we use that the Weyl quantization of the constant $1$ function is the identity operator to write
\begin{equation}\label{scm.mass1.1}
	1 = \int_{\R} |u_{k_j}(x)|^2 \,\dd x = \int_{\R} \overline{u_{k_j}(x)}\left(b^w(x, \hbar_k{k_j}D_x) + (1-b)^w(x,\hbar_{k_j} D_x)\right) u_{k_j}(x) \,\dd x.
\end{equation}
By \cite[Thm.~6.4]{Zworski-2012},
\begin{equation}\label{scm.mass1.2}
	(1-b)^w (x,\hbar_{k_j} D_x)u_k(x) = \BigO(\hbar_{k_j}^\infty),
\end{equation}
meaning that its $L^2(\R)$ norm is smaller than any power of $\hbar_{k_j}$ as $\hbar_{k_j} \to 0^+$, and by the definition \eqref{scm.mu.limit} of $\mu(x, \xi)$ and the fact that $b \equiv 1$ on $\operatorname{supp} \mu$,
\begin{equation}\label{scm.mass1.3}
	\lim_{k \to \infty} \int_{\R} \overline{u_k(x)}b^w(x, \hbar_k D_x)u_k(x)\,\dd x = \int_{\R^2} b(x,\xi)\,\dd \mu (x, \xi) = \int_{\R^2} \dd \mu(x,\xi).
\end{equation}
Taking \eqref{scm.mass1.1}, \eqref{scm.mass1.2}, and \eqref{scm.mass1.3} together proves \eqref{scm.mass1}.

We now prove that a measure $\mu$ satisfying the properties of a semiclassical defect measure must have the form matching the classical heuristic \eqref{scm.classicalheuristic} generalized in Theorem \ref{thm:lim.EF}.

\begin{proposition}
Let $V(x) \in C^\infty(\R; \R)$ satisfy \eqref{scm.Vsymbol}, \eqref{scm.Vgrowth}, and \eqref{scm.Vwell}. Let $\lambda_0 > V(0) = \inf V(x)$, and let $\mu$ be a measure satisfying \eqref{scm.suppmu}, \eqref{scm.bracket}, and \eqref{scm.mass1} and let $\w$ be as in \eqref{rho.def}. Then the measure $\mu$ obeys for all $b \in C_c^\infty(\R^2)$
\[
	\int b(x,\xi)\,\dd \mu = c_0\int_{-x_{\lambda_0}}^{x_{\lambda_0}} 
	\left( b(x, \w(x))+b(x, -\w(x)) \right) \frac{\dd x}{\w(x)},
\]
where the normalization constant $c_0$ is such that $\int \dd \mu = 1$. 
\end{proposition}
\begin{proof}
We observe that
\begin{equation}\label{scm.xdiff}
	\begin{aligned}
	\frac{\dd}{\dd x}b(x,\w(x)) &= b_x(x,\w(x)) + \w'(x)b_\xi(x,\w(x))
	\\ &= b_x(x,\w(x)) - \frac{V'(x)}{2\w}b_\xi(x,\w(x))
	\\ &= \frac{1}{2\w(x)}\left(2\w(x) b_x(x,\w(x)) - V'(x)b_\xi(x,\w(x))\right)
	\\ &= \frac{1}{2\w(x)}\{a, b\}(x,\w(x)).
	\end{aligned}
\end{equation}
Letting $b\in C_c^\infty(\R^2)$ be such that $\supp b \subset \{\xi > \delta\}$ for some $\delta > 0$, we obtain from \eqref{scm.nupush}, \eqref{scm.bracket}, and \eqref{scm.xdiff} that
\[
	\int \left(\frac{\dd}{\dd x}b(x,\w(x))\right) 2\w(x)\,\dd \nu_+(x)
\]
vanishes. Taking $b(x,\xi) = f(x)\chi_{[\delta, \delta^{-1}]}(\xi)$ for $f\in C_c^\infty(\R)$ arbitrary and for $\chi$ a cutoff function, letting $\delta \to 0^+$ allows us to conclude that
\[
	\int f'(x)\,\w(x)\,\dd \nu_+(x) = 0
\]
for all $f\in C_c^\infty(\R)$. Therefore along $\{\xi^2 + V(x) = \lambda_0\}$,
\[
\dd \nu_+(x) = \frac{c_+}{\w}\dd x
\]
for some $c_+$ which is positive because $\mu$ is a positive measure.

When $\supp b(x, \xi) \subset \{\xi < 0\}$, the same argument shows that there is some $c_- > 0$ such that 
\[
	\int b(x,\xi) \,\dd \mu(x,\xi) = \int b(x,-\w(x))\,\frac{c_-}{\w}(x)\dd x.
\]

One can show then that $c_+ = c_-$ by projecting onto the $\xi$ variable instead of the $x$ variable: let 
\[
	\tilde{x}(\xi) = V^{-1}(\lambda_0 - \xi^2)
\]
where the inverse image is chosen positive, and let $\dd \rho_+(\xi)$ be such that when $\supp b \subset \{x > 0\}$,
\[
	\int b(x,\xi)\,\dd\mu(x,\xi) = \int b(\tilde{x}(\xi), \xi)\,\dd \rho_+(\xi).
\]
Then $\tilde{x}'(\xi) = -\frac{2\xi}{V'(\tilde{x}(\xi))}$,
\[
	\frac{\dd}{\dd x}b(\xi(x), x) = V'(\tilde{x}(\xi))\{a, b\}(\tilde{\xi}(x), x).
\]
The earlier argument (along with the fact that $V'(x) > 0$ for $x > 0$) shows that there is some $d_+ > 0$ such that
\[
	\dd \rho_+(\xi) = \frac{d_+}{V'(\tilde{x}(\xi))}\,\dd \xi.
\]

On $\{\xi^2 + V(x) = \lambda_0\}$, note that
\[
	\left|\frac{\dd \xi}{\dd x}\right| = \frac{|V'(x)|}{2|\tilde{\xi}(x)|}.
\]
Since the pull-backs of $\dd \nu_\pm(x) = \frac{c_\pm}{|\xi|}\dd x$ and $\dd \rho_+$ agree on $a^{-1}(\{\lambda_0\}) \cap \{x>0, \xi>0\}$ and since $\dd \rho_+$ and $\dd \nu_-$ agree on $\{x > 0, \xi < 0\}$ we can conclude that $c_+ = d_+ = c_-$. We remark that this argument is not available in the case $\omega_\beta = 0$ corresponding to a very rapidly-growing potential.

Finally, we conclude that $c_0 = c_+$ is such that $\int \dd \mu = 1$ by the hypothesis \eqref{scm.mass1}.
\end{proof}


{\footnotesize
\bibliographystyle{acm}
\bibliography{C:/Data/00Synchronized/references}
}

\end{document}

%% file: commands.tex
\usepackage{amsmath,amsthm,amssymb}
\usepackage{mathrsfs}
\usepackage[shortlabels]{enumitem}
\usepackage{graphicx}
\usepackage[font=small]{caption}
\usepackage{xcolor}
\usepackage{url}

\newcommand{\N}{\mathbb{N}}
\newcommand{\R}{{\mathbb{R}}}

\newcommand{\Z}{{\mathbb{Z}}}

\newcommand{\dd}{{{\rm d}}}
\newcommand{\ii}{{\rm i}}

\newcommand{\cf}{\emph{cf.}}
\newcommand{\ie}{{\emph{i.e.}}}
\newcommand{\eg}{{\emph{e.g.}}}

\newcommand{\la}{\lambda}

\newcommand{\eps}{\varepsilon}

\newcommand{\esssup}{\operatorname*{ess \,sup}}

\newcommand{\Dom}{{\operatorname{Dom}}}

\newcommand{\supp}{\operatorname{supp}}

\newcommand{\loc}{\mathrm{loc}}
\newcommand{\BigO}{\mathcal{O}}

\newcommand{\ls}{\lesssim}
\newcommand{\gs}{\gtrsim}

\theoremstyle{plain}

\newtheorem{theorem}{Theorem}[section]
\newtheorem{lemma}[theorem]{Lemma}

\newtheorem{proposition}[theorem]{Proposition}

\theoremstyle{definition}

\newtheorem{asm}{Assumption}

\newcommand\cF{\mathcal F}

\newcommand\cI{\mathcal I}
\newcommand\cJ{\mathcal J}